\documentclass[a4paper,11pt]{article}
\usepackage{amsthm,amsmath,amssymb,amsfonts,bbm,graphics,a4wide,rotating,graphicx}
\usepackage{color}
\usepackage{natbib}
\def\Var{\mathrm{Var}}

\newcommand{\indic}{\mathbbm{1}}
\newtheorem{Th}{Theorem}

\newtheorem{Lemma}{Lemma}

\renewenvironment{proof}{\noindent{\bf Proof.}}{\hfill
  $\blacksquare$\par\noindent}

\newcommand{\norm}[1]{\ensuremath{\vert\!\vert #1 \vert\!\vert}}
\newcommand{\E}{\ensuremath{\mathbb{E}}}

\renewcommand{\P}{\ensuremath{\mathbb{P}}}

\newcommand{\R}{\ensuremath{\mathbb{R}}}

\renewcommand{\L}[1]{\ensuremath{\mathbb{L}^{#1}}}
\newcommand{\e}{\ensuremath{\varepsilon}}
\newcommand{\la}{\ensuremath{{jk}}}
\newcommand{\p}{\ensuremath{\psi}}

\newcommand{\Supp}{\mathrm{Supp}}

\newcommand{\Z}{\ensuremath{\mathbb{Z}}}

\renewcommand{\ln}{{\log\,}}
\renewcommand{\L}{\ensuremath{\mathbb{L}}}
\newcommand{\La}{\ensuremath{\Lambda}}
\newcommand{\Ga}{\ensuremath{\Gamma}}
\newcommand{\ga}{\ensuremath{\gamma}}
\newcommand{\al}{\ensuremath{\alpha}}

\newcommand{\si}{\ensuremath{\sigma}}

\newcommand{\tp}{\ensuremath{{\tilde{\psi}}}}
\newcommand{\be}{\ensuremath{\beta}}
\newcommand{\tb}{\ensuremath{\tilde{\beta}}}
\newcommand{\hb}{\ensuremath{\hat{\beta}}}
\newcommand{\jk}{\ensuremath{{jk}}}

\numberwithin{equation}{section}

\begin{document}

\title{Adaptive density estimation: a curse of support?}
\title{Adaptive density estimation: a curse of support?}
\author{Patricia REYNAUD-BOURET \thanks{CNRS, Laboratoire Jean-Alexandre Dieudonn\'e, CNRS UMR 6621, Universit\'e de Nice Sophia-Antipolis, Parc Valrose, 06108 Nice Cedex 2, France. E-mail: reynaudb@unice.fr} \thanks{Corresponding author: tel: (+33) 04 92 07 60 33; fax: (+33) 04 93 51 79 74}~, Vincent RIVOIRARD\thanks{Laboratoire de Math\'ematique, CNRS 8628, Universit\'e de Paris Sud, 91405 Orsay Cedex, France. D\'epartement de Math\'ematiques et Applications, ENS-Paris, 45 rue d'Ulm, 75230 Paris Cedex 05, France. E-mail: Vincent.Rivoirard@math.u-psud.fr} \ and Christine TULEAU-MALOT\thanks{Laboratoire Jean-Alexandre Dieudonn\'e, CNRS UMR 6621, Universit\'e de Nice Sophia-Antipolis, Parc Valrose, 06108 Nice Cedex 2, France. E-mail: malot@unice.fr}}


\maketitle

\begin{abstract}
This paper deals with the classical problem of density estimation on the real line. Most of the existing papers devoted to minimax properties assume that the support of the underlying density is bounded and known. But this assumption may be very difficult to handle in practice. In this work, we show that, exactly as a curse of dimensionality exists when the data lie in $\R^d$, there exists a curse of support as well when the support of the density is infinite. As for the dimensionality problem where the rates of convergence deteriorate when the dimension grows, the minimax rates of convergence may deteriorate as well when the support becomes infinite. This problem is not purely theoretical since the simulations show that the support-dependent methods are really affected in practice by the size of the density support, or by the weight of the density tail. We propose a method based on  a biorthogonal wavelet thresholding rule that is adaptive with respect to the nature of the support and the regularity of the signal, but that is also robust in practice to this curse of support. The threshold, that is proposed here, is very accurately calibrated so that the gap between optimal theoretical and practical tuning parameters is almost filled.
\end{abstract}


~

\section{Introduction}
This paper deals with the classical problem of density estimation for unidimensional data. Our aim is to provide an adaptive
method which requires as few assumptions as possible on the underlying density in order to apply it in an exploratory
way. In particular, we do not want to have any assumption on the density support.
 Moreover this method should be quite easy to implement and should have good theoretical performance as well.

Density estimation is a task that lies at the core of many
 data preprocessing. From this point
of view, no assumption should
be made on the underlying function to estimate. Without giving a full survey of the subject, let us describe classical methods of the literature.

At least in a first approach, histograms or kernel
methods are often used. The main problem is to choose the bandwidth (see  for instance \citet{silver1}), which is usually performed by
cross-validation (see the fundamental paper by \citet{rud}). There is no clear theoretical results about the
performance of this cross-validation method from an adaptive minimax point of view.
 However, methodologies based on
kernel methods are the most widespread in practice. Consequently, several data-driven methods have been developed  (see
\citet{silver} for a good review). There exist fundamentally two ways to extend  cross-validation. The first one is to obtain more competitive methods from the computational point of view (see  \citet{GrayMoore}). The second one is to find more robust and less undersmoothing methods (see \citet{JMS} for a recent survey).

All these methods suffer from a lack of spatial adaptivity since the bandwidth is selected uniformly in space. 
To improve this point, \citet{SainScott}  have suggested a practical kernel method which makes the choice of the bandwidth more local, this algorithm being still based on intensive cross-validation. All these methods do not require in practice the preliminary knowledge of the support but do not provide theoretical guarantees from the minimax point of view. On the contrary, in the white noise model, under assumptions on the underlying signal and its support, it is possible to select the best possible local bandwidth in the adaptive minimax setting via the Lespki method (see \citet{LepMS} for instance).

The Lepski method is closely related to model selection methods.
Following Akaike's criterion for histograms, \citet{gwen}   has derived adaptive minimax procedures for density estimation (see \citet{stflourpascal} for detailed proofs and \citet{rozebir} for a practical point of view). To remedy the lack of smoothness of histograms,
 piecewise polynomial estimates can also be used (see for instance \citet{gwen2} and \citet{roze} for the corresponding software, \citet{wn} or \citet{KKP} for the  spline basis). It is worth emphasizing that the necessary input of all these methods is the support of the underlying density, classically assumed to be $[0,1]$. We can also cite the results based on $\ell_1$ penalties. See \citet{bunea}, \citet{bunea2} and \citet{bertin} who derived oracle inequalities for which no assumptions on the support are made. However, minimax optimality is not investigated in these papers and for simulations, 
\citet{bertin} considered signals supported by $[0,1]$.   So, whether the
support plays a key role for $\ell_1$ methodologies remains an open
question even if we naturally conjecture that the answer is yes. In practice, the data are usually rescaled by the smallest and largest observations before performing any of the previous algorithms. This preprocessing has not been studied theoretically. In particular, what happens if the density is heavy-tailed?

Now let us turn to  wavelet thresholding. \citet{djkp} have first provided theoretical adaptive minimax results in the density setting. This paper is a theoretical benchmark but  their threshold depends on the extraknowledge of the infinite norm of the underlying density. In practice, even if this quantity is known, this choice is often too conservative. From a computational point of view, the DWT algorithm due to \citet{Mal} combined with a keep or kill rule on each coefficient makes these methods as one of the easiest adaptive
 methods to implement, once the threshold is known. Here lies the fundamental problem: after rescaling and binning the data as in \citet{anton} for instance, one can reasonably think that the number of  observations in a ``not too small'' interval is Gaussian, up to some eventual transformation  (see \citet{cai}).
So basically the thresholding rules adapted to
 the Gaussian regression setting
 should work here (we refer the reader to the very complete review paper of \citet{abs}  which provides descriptions and comparisons of various wavelet shrinkage and thresholding estimators in the regression setting).
 Of course many assumptions are required.
Even if in \citet{cai}, theoretical justifications are given,
 the method still relies heavily on the precise knowledge of the support which is directly linked to the size of the bins.
In their seminal work \citet{silverHN} have already observed that in practice the basic Gaussian approximation for general wavelet bases was quite poor. This can be corrected by the use of the Haar basis and accurate thresholding rules but the reconstructions are consequently piecewise constant.
Note also that in \citet{silverHN} no assumption was made on the possible support of the underlying density.
More recently, \citet{jll} have proposed an adaptive thresholding procedure on the whole real line. Their threshold is not based on a direct Gaussian approximation. Indeed, the chosen threshold depends randomly on the localization in time and frequency of the coefficient that has to be kept or killed. They derive adaptive minimax results for H\"olderian spaces, exhibiting rates that are different from the bounded support case. However 
there is  a gap between their optimal theoretical and practical tuning parameters of the threshold.

If the main goal of this paper is to investigate assumption-free wavelet thresholding methodologies as explained in the first paragraph, we also aim at fulfilling this gap by designing a new threshold depending on a tuning parameter $\gamma$: the precise form of the threshold  is closely related to sharp exponential inequalities for iid variables, avoiding the use of Gaussian approximation. Unlike methods of \citet{jll} and \citet{silverHN}, all the coefficients (and in particular the coarsest ones) are likely to be thresholded.
Moreover, since our threshold is defined very accurately from a non asymptotic point of view, we obtain sharp oracle inequalities for $\gamma>1$. But we also prove that taking $\gamma<1$ deteriorates the theoretical properties of our estimators. Hence the remaining gap between theoretical and practical thresholds lies in a second order term (see Section \ref{method} for more details). The construction of our estimators and the previous results are stated in Section 2. Next, in Section 3, we illustrate the impact of the bounded support assumption by exhibiting minimax rates of convergence on the whole class of Besov spaces extending the results of \citet{jll}. In particular, when the support is infinite, our results reveal how minimax rates deteriorate according to the sparsity of the density.   We also show that our estimator is  adaptive minimax (up to a logarithmic term) over Besov balls with respect to the regularity  but also with respect to the support (finite or not). In Section 4, we investigate the curse of support for the most well-known support-dependent methods and compare them with our method and with the cross-validated kernel method. Our method, which is naturally spatially adaptive, seems to be robust with respect to the size of the support or the tail of the underlying density. We also implement our method on real data, revealing the potential impact of our methodology for practitioners. The appendices are dedicated to an analytical description of the biorthogonal wavelet basis but also to the proofs of the main results.

\section{Our method}\label{method}
Let us observe a $n$-sample of density $f$ assumed to be in $\L_2(\R)$. We denote this sample $X_1,\dots,X_n$.
We estimate $f$ via its coefficients  on a special biorthogonal wavelet basis, due to \citet{cdf}. The decomposition of $f$ on such a basis takes the following form:
\begin{equation}\label{decom2}
f=\sum_{k\in\Z}\beta_{-1k}\tilde\psi_{-1k}+\sum_{j\geq
0}\sum_{k\in\Z}\beta_{jk}\tilde\psi_{jk},
\end{equation}
where for any  $j\geq 0$ and any $k\in\Z$,
$$\beta_{-1k}=\int_\R          f(x)\psi_{-1k}(x)dx,\quad          \beta_{jk}=\int_\R
f(x)\psi_{jk}(x)dx.$$
The most basic example of biorthogonal wavelet basis is the Haar basis where the father wavelets are given by $$\forall k \in \Z, \quad \psi_{-1k}=\tilde{\psi}_{-1k}=\indic_{[k;k+1]}$$ and the mother wavelets are given by $$\forall j \geq 0,\ \forall k \in \Z, \quad \psi_{jk}=\tilde{\psi}_{jk}=2^{j/2}\left(\indic_{[k2^{-j};(k+1/2)2^{-j})}-\indic_{[(k+1/2)2^{-j};(k+1)2^{-j}]}\right).$$
The other examples we consider are more precisely described in Appendix A. 
The essential feature is that it is possible to use, on  one hand, decomposition wavelets $\psi_{jk}$ that are piecewise constants, and, on the other hand, smooth reconstruction wavelets $\tilde{\psi}_{jk}$. In particular, except for the Haar basis, decomposition and reconstruction wavelets are different.
To shorten mathematical expressions, we set
\begin{equation}\label{Lajk}
\Lambda=\{(j,k):\quad j\geq -1,k\in\Z\}
\end{equation}
and   (\ref{decom2}) can be rewritten as
\begin{equation}
\label{decom}
f=\sum_{(j,k)\in \La}\be_\la \tp_\la\quad \mbox{  with }\quad \be_\la=\int \p_\la(x) f(x)
dx.
\end{equation}
A classical unbiased estimator for $\be_\la$ is the empirical coefficient
\begin{equation}
\label{defest1}
\hb_\la=\frac{1}{n}\sum_{i=1}^n \p_{jk}(X_i),
\end{equation}
whose variance is $\si^2_{jk}/n$ where
$$\si^2_{jk}=\int \p^2_{jk}(x) f(x)dx-\left(\int \p_{jk}(x) f(x)dx\right)^2.$$
Note that $\si^2_{jk}$ is classically unbiasedly estimated by
$\widehat{\si}^2_{jk}$ with
$$\widehat{\si}^2_{jk}=\frac{1}{n(n-1)}\sum_{i=2}^n\sum_{l=1}^{i-1}(\p_{jk}(X_i)-\p_{jk}(X_l))^2.$$
Now, let us define our  thresholding  estimate of $f$.  In the sequel there are two different kinds of steps, depending on whether the estimate is used for theoretical or practical purposes. Both situations are respectively denoted '{\it Th.}' and '{\it Prac}'.
\begin{itemize}
\item[{\bf Step 0~~}] ~
  \begin{itemize}
  \item[{\it Th.~~~}] Choose a constant $c\geq 1$, a real number $c'$ and let $j_0$ such that ${j_0=\lfloor\log_2([n^c(\ln n)^{c'}])\rfloor}$. Choose also a positive constant $\gamma$.
  \item[{\it Prac.~}] Let $j_0=\lfloor\log_2(n)\rfloor$.
  \end{itemize}
\item[{\bf Step 1~~}] Set $\Gamma_n=\{(j,k):\ -1\leq j\leq j_0,\ k\in\Z\}$ and compute for any $(j,k)\in\Gamma_n$, the non-zero empirical coefficients
$\hb_\la$ (whose number is almost surely finite).
\item[{\bf Step 2~~}] Threshold the coefficients by setting $\tb_\la=\hb_\la\indic_{|\hb_\la|\geq \eta_\la}$
  according to the following threshold choice.
\begin{itemize}
  \item[{\it Th.~~~}] Overestimate slightly the variance $\si^2_\la$ by
$${\widetilde{\si}}^2_{jk}=\widehat{\si}^2_{jk}+2\norm{\p_{jk}}_\infty\sqrt{2\gamma\widehat{\si}^2_{jk}\frac{\ln n}{n}}+8\gamma \norm{\p_{jk}}_\infty^2 \frac{\ln n}{n}$$
and choose
\begin{equation}\label{defthresh}
\eta_{\la}=\eta_{\la,\ga}=\sqrt{2\gamma {\widetilde{\si}}^2_{jk}\frac{\ln n}{n}} +\frac{2\norm{\p_{jk}}_\infty\gamma \ln n}{3n}.
\end{equation}
  \item[{\it Prac.~}] Estimate unbiasedly the variance by $\widehat{\si}^2_{jk}$ and choose
\begin{equation}\label{defthreshprac}
\eta_{\la}=\eta_{\la}^{Prac}=\sqrt{2\widehat{\si}^2_{jk} \frac{\ln n}{n}} +\frac{2\norm{\p_{jk}}_\infty \ln n}{3n}.
\end{equation}
\end{itemize}
\item[{\bf Step 3~~}] Reconstruct the function by using the $\tb_\la$'s and denote
\begin{itemize}
  \item[{\it Th.~~~}] \begin{equation}\label{defest}
\tilde{f}_{n,\gamma}=
\sum_{(j,k) \in \Gamma_n} \tb_\la \tp_\la
\end{equation}
  \item[{\it Prac.~}]
\begin{equation}\label{defestbis}
\tilde{f}_{n}^{Prac}=\left(
\sum_{(j,k) \in \Gamma_n} \tb_\la \tp_\la\right)_+
\end{equation}
\end{itemize}
\end{itemize}

Note that this method can easily be implemented with a low computational cost. In particular, unlike the DWT-based algorithms, our algorithm does not need numerical approximations, except at {\bf Step 3} for the computations of the $\tilde\p_{jk}$ (unless, we use the Haar basis). However, a preprocessing, independent of the algorithm, can be used to compute reconstruction wavelets at any required precision.
Both practical and theoretical thresholds are based on the following heuristics.
Let $c_0>0$. Define the heavy mass zone as the set of indices $(j,k)\in\Lambda$  such that $f(x)\geq
c_0$ for $x$ in the support of $\p_\la$ and
 $\|\p_\la\|_{\infty}^2=o_n(n(\log   n)^{-1})$.
In this heavy mass zone, the random term of (\ref{defthresh}) or (\ref{defthreshprac}) is the main one and we asymptotically derive that with large probability
\begin{equation}\label{approxseuil}
\eta_{\la,\ga}\approx\sqrt{2\gamma \widetilde{\sigma}^2_\la\frac{\ln n}{n}} \quad\mbox{ and }\quad \eta_{\la}^{Prac}\approx\sqrt{2 \widehat{\sigma}^2_\la\frac{\ln n}{n}}.
\end{equation}
The shape of the right hand terms in (\ref{approxseuil}) is classical in the density estimation framework (see \citet{djkp}). In fact, they  look like the threshold proposed by \citet{jll} or the universal threshold $\eta^U$ proposed by \citet{dojo} in the Gaussian regression  framework. Indeed, we recall that, in this set-up,
$$\eta^U=\sqrt{2\sigma^2\log{n}},$$ where  $\sigma^2$  (assumed
to be known in the Gaussian framework) is the variance of each noisy wavelet coefficient.
Actually, the deterministic term of
(\ref{defthresh})  (or (\ref{defthreshprac})) constitutes the main difference with the threshold proposed by \citet{jll}: it replaces the second keep or kill rule applied by Juditsky and Lambert-Lacroix on the empirical coefficients. This additional term allows to control large deviation terms for high resolution levels. It is directly linked to Bernstein's inequality (see the proofs in  Appendix B). The forthcoming oracle inequality (Theorem \ref{inegoraclelavraie})  holds with (\ref{defthresh}) for any $\gamma>1$: this is essential to fulfill the gap between theory and practice. Indeed, note that if one takes $c=\gamma=1$ and $c'=0$ then the main difference between (\ref{defthresh}) and  (\ref{defthreshprac}) is that a second order term exists in the estimation of $\si^2_\la$ by $\widetilde{\sigma}^2_\la$. But the main part is exactly the same: when the coefficient lies in the heavy mass zone and when $\gamma$ tends to 1, $\eta_{\la,\ga}$ tends to $\eta_{\la}^{Prac}$ with high probability. Indeed, one can note that for all $\e>0$ and $\gamma>1$,
$$\eta_{\la}^{Prac}\leq\eta_{\la,\ga}\leq \sqrt{2\gamma(1+\e)\widehat{\si}^2_{jk} \frac{\ln n}{n}} +\left(\frac{2}{3}+\sqrt{8+2\e^{-1}}\right)\frac{\norm{\p_{jk}}_\infty\gamma \ln n}{n}.$$

As often suggested in the literature, instead of estimating $\Var(\hb_\la)$, we could have used the inequality
$$\Var(\hb_\la)=
\frac{{\si}^2_{jk}}{n}\leq \frac{\norm{f}_\infty}{n}$$
and we could have replaced ${\widetilde{\si}}^2_{jk}$ with $\norm{f}_\infty$ in the definition of the threshold. But this requires a strong assumption: $f$ is bounded and  $\norm{f}_\infty$ is known. In our paper, $\Var(\hb_\la)$ is accurately  estimated making those conditions unnecessary. 
 Theoretically,  we slightly overestimate $\si^2_{jk}$ to control large deviation terms and this is the reason why we introduce ${\widetilde{\si}}^2_{jk}$. Note that \citet{calib} have proposed thresholding rules based on similar heuristic arguments in the Poisson intensity estimation framework. But proofs and computations are more involved for density estimation because sharp upper and lower bounds for $\widehat{\si}^2_{jk}$ are more intricate.

For practical purpose, $\eta_{\la,\ga}$ (even with $\ga=1$) slightly oversmooths the estimate with respect to $\eta_{\la}^{Prac}$. From a simulation point of view, the linear term $\frac{2\norm{\p_{jk}}_\infty \ln n}{3n}$ in $\eta_{\la}^{Prac}$ with the precise constant $2/3$ seems to be accurate.

The remaining part of this section is dedicated to a precise choice of $\gamma$, first from an oracle point of view, next from a theoretical and practical study.

\subsection{Oracle inequalities}
The oracle point of view has been introduced by \citet{dojo}. In this approach, an estimate is optimal if it can essentially mimic the performance of the ``oracle estimator''. Let us recall that the latter is not a true estimator since it depends on the function to be estimated but it represents an ideal for a particular method (namely, here, wavelet thresholding). So, in our framework, the oracle provides the noisy wavelet coefficients that have to be kept. It is easy to see  that the ``oracle estimate'' is $$\bar{f_n}=\sum_{(j,k)\in \Ga_n}
\bar{\be}_\la\tp_\la,$$ where
$\bar{\be}_\la=\hb_\la \indic_{\{\be_\la^2>{\si}^2_{jk}/n\}}$
satisfies
$$\E\left[(\bar{\be}_\la-{\be}_\la)^2\right]=\min\left(\be_\la^2,\frac{{\si}^2_{jk}}{n}\right).$$
By keeping the coefficients $\hb_\la$ larger than the thresholds defined in (\ref{defthresh}),  our estimator has a risk that is not larger than the oracle risk,  up to a logarithmic term, as stated by the following result.
\begin{Th}\label{inegoraclelavraie}
Let us consider a biorthogonal wavelet basis satisfying the properties described in Appendix A.  If $\ga  >c$,  then $\tilde  f_{n,\ga}$ satisfies  the
following oracle inequality: for $n$ large enough
\begin{equation}\label{inegoraclelavraie1}
\E\left[\norm{\tilde{f}_{n,\ga}-f}_2^2\right]\leq
C_1\left[\sum_{(j,k)\in\Gamma_n}\min\left(\be_\la^2,\ln n\frac{{\si}^2_{jk}}{n}\right)+\sum_{(j,k)\notin\Gamma_n}\be_\la^2\right]+\frac{C_2\log n}{n}
\end{equation}
where $C_1$  is a positive  constant depending  only on
$\ga$, $c$ and the choice of the wavelet basis 
and where $C_2$ is also a
positive constant depending on
$\ga$, $c$, $c'$, 
$\|f\|_2 $ and the choice of the wavelet basis.
\end{Th}
Note that Theorem \ref{inegoraclelavraie} holds with $c=1$ and $\ga>1$, as announced.
 Following the oracle point of view of Donoho and Johnstone, Theorem
\ref{inegoraclelavraie} shows that our procedure is optimal up to the logarithmic factor (and the negligible term $\log n/n$). This logarithmic term is in some sense unavoidable. It is the price we pay for adaptivity, i.e. the fact that we do not know the coefficients to keep. Note also that our result is true provided $f\in\L_2(\R)$. So, assumptions on $f$ are very mild here. This is not the case for
most of the results for non-parametric estimation procedures where one assumes that $\norm{f}_\infty<\infty$ and that $f$ has a compact support. Note in addition that this support and $\norm{f}_\infty$ are often known in the literature. On the contrary, in Theorem \ref{inegoraclelavraie}, $f$ and its support can be unbounded. So, we make as few assumptions as possible. This is allowed by considering random thresholding with the data-driven thresholds defined in (\ref{defthresh}).

\subsection{Calibration issues}
We address the problem of choosing conveniently the threshold parameter $\gamma$ from the theoretical point of view. The aim and the proofs are inspired by \citet{bm} who considered penalized estimators and calibrated constants for penalties in a Gaussian framework. In particular, they showed that if the penalty constant is smaller than 1, then the penalized estimator behaves in a quite unsatisfactory way. This study was used in practice to derive adequate data-driven penalties by \citet{leb}.

According to Theorem \ref{inegoraclelavraie}, we notice that 
 for any signal, taking $c=1$ and $c'=0$, we achieve the oracle performance up to a logarithmic term provided $\gamma>1$. So, our primary interest is to wonder what happens, from the theoretical point of view, when $\ga\leq 1$?

 To handle this problem, we consider the simplest signal in our setting and we compare the rates of convergence when $\ga>1$ and $\ga<1$.
\begin{Th}\label{lower}
Let $f=\indic_{[0,1]}$ and let us consider $\tilde{f}_{n,\ga}$ with the Haar basis, $c=1$ and $c'=0$.
\begin{itemize}
\item If $\ga>1$ then there exists a constant $C$ depending only on $\ga$ such that
$$\E(\norm{\tilde{f}_{n}-f}^2)\leq C\frac{\ln n}{n}.$$
\item If $\ga<1$, then there exists $\delta<1$ depending only on $\ga$ such that
$$\E(\norm{\tilde{f}_{n}-f}^2_2)\geq \frac{1}{n^{\delta}}(1+o_n(1)).$$
\end{itemize}
\end{Th}
Theorem \ref{lower}   establishes     that,    asymptotically,
$\tilde{f}_{n,\ga}$       with      $\gamma<1$       cannot      estimate a very simple signal ($f=\indic_{[0,1]}$)  at a convenient rate  of convergence.  This  provides a
lower bound for the threshold parameter $\gamma$: we have to take  $\ga\geq 1$.

\begin{figure}[h!]
\begin{center}
\hspace{-0.5cm} \includegraphics[width=0.8\linewidth,angle=0]{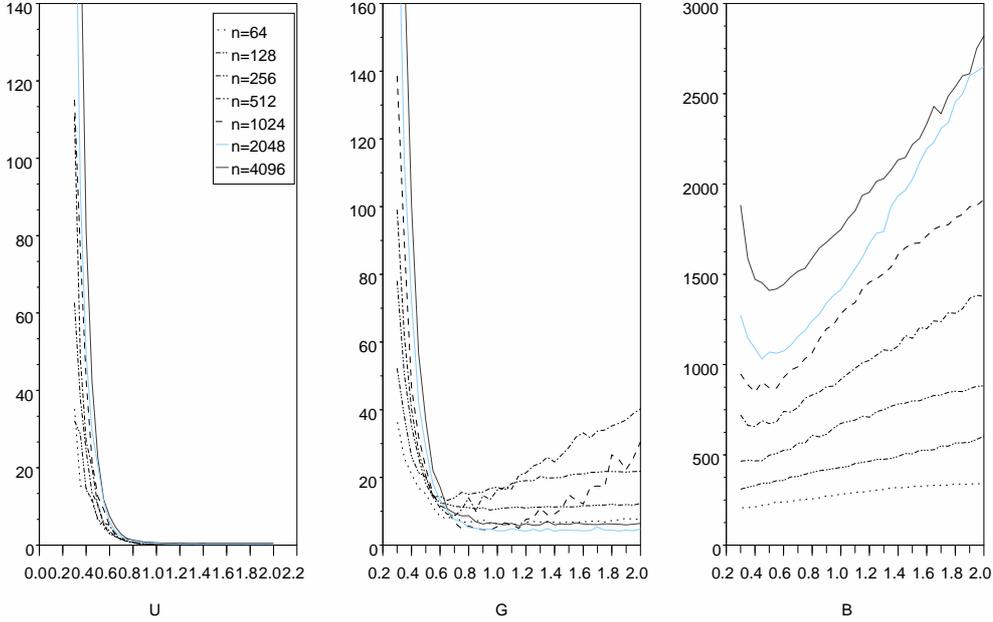}
\end{center}
\caption{$n\times MISE_n(\ga)$  for ({\bf U}) $f=\indic_{[0,1]}$ (the Haar basis is used) ; ({\bf G}) $f$ Gaussian density with mean 0.5 and standard deviation 0.25  (the Spline basis
is used) ; ({\bf B}) $f$ is the renormalized Bumps signal (the Spline basis is used)\label{tracegamma}
}
\end{figure}

We reinforce these results by a simulation study. First we
simulate 1000 n-samples of density $f=\indic_{[0,1]}.$ We
estimate $f$ by $\tilde{f}_n^{Prac}$ using the Haar basis, but to see the influence of the parameter $\gamma$ on the estimation, we
replace $\eta_{jk}^{Prac}$ (see {\bf{Step 2}} (\ref{defthreshprac})) by
\begin{equation}\label{defthreshpracavecgamma}
\eta_{\la}=\sqrt{2\gamma\widehat{\si}^2_{jk} \frac{\ln n}{n}} +\frac{2\gamma\norm{\p_{jk}}_\infty \ln n}{3n}.
\end{equation}
For any $\gamma$, we have
 computed $MISE_n(\gamma)$ i.e. the average over the 1000 simulations of $\norm{\tilde{f}_n^{Prac}-f}^2$.
On the left part of Figure  \ref{tracegamma} ({\bf U}),  $MISE_n(\gamma)\times n$ is plotted as a
  function of $\gamma$ for different values of $n$. Note that when $\gamma>1$, $MISE_n(\gamma)$ is null meaning that our procedure selects just one wavelet coefficient, the one associated to $\psi_{-1,0}=\indic_{[0,1]}$; all others are equal to zero. This fact remains true for a very large range of values of $\gamma$. This plateau phenomenon has already been noticed in the Poisson framework (see \citet{calib}). However as soon as $\gamma<1$,  $MISE_n(\gamma)\times n$ is positive and increases when $\gamma$ decreases. It also increases with $n$ tending to prove that $MISE_n(\gamma) >> 1/n$ for $\gamma<1$. This is in complete adequation with Theorem \ref{lower}. Remark that, from a theoretical point of view, the proof of part 2 of Theorem \ref{lower} holds for any choice of threshold that is asymptotically equivalent to $\sqrt{2\gamma\widehat{\si}^2_{jk} \frac{\ln n}{n}}$ in the heavy mass zone and in particular for the choice (\ref{defthreshpracavecgamma}). From a numerical point of view, the  left part of Figure \ref{tracegamma} ({\bf U}) would have been essentially the same with 
  $\eta_{\la,\ga}$, i.e. (\ref{defthresh}) instead of (\ref{defthreshpracavecgamma}).
The reason why we
used (\ref{defthreshpracavecgamma}) is the practical performance when the function $f$ is more irregular with respect to the chosen basis. Indeed we 
consider two other density functions $f$. The first one is the density of a Gaussian variable whose results appear in the middle part of Figure \ref{tracegamma} ({\bf G}) and the second one is the renormalized Bumps signal \footnote{~ The renormalized Bumps signal is a very irregular signal that is classically used in wavelet analysis. It is here renormalized so that the integral equals 1 and it can be defined by $\displaystyle\left(\sum_jg_j\left(1+\frac{|x-p_j|}{w_j}\right)^{-4}\right)\frac{{\bf 1}_{[0,1]}}{0.284}$ with
\begin{center}
\begin{tabular}{ccccccccccccccc}
  p & =  & [ & 0.1 & 0.13 & 0.15 &0.23 &0.25& 0.4  & 0.44 & 0.65 & 0.76 &0.78  & 0.81 &  ]  \\
 g  & = & [ & 4  & 5 & 3 & 4 & 5 & 4.2 & 2.1 & 4.3  &3.1  & 5.1  & 4.2  & ] \\
  w&=  &[  &0.005  & 0.005 &0.006  &0.01  &0.01  &0.03  &0.01  &0.01  &0.005  &0.008  &0.005  & ]
\end{tabular}
\end{center}
} whose results appear in the right part
of Figure \ref{tracegamma} ({ \bf B}). In both cases we computed $\tilde{f}_n^{Prac}$ with the Spline basis : this basis is a particular possible choice of the wavelet basis which leads to smooth estimates. A description is available in Figure \ref{Spline} of Appendix A. We computed the associate
$MISE_n(\gamma)$
over 100 simulations.
Note that for the Bumps signal, there is no plateau phenomenon and that the best choice for $\gamma$ is $\gamma=0.5$ as soon as the highest level of resolution, $j_0(n)$ is high enough to capture the irregularity of the signal. If $n$ is too small, the best choice is to keep all the coefficients.
As already noticed in \citet{calib}, there exists in fact two behaviors : either the oracle $\bar{f}_n$ is close to $f$ and the best possible choice is $\gamma\simeq 1$ with a plateau phenomenon, or the oracle $\bar{f}_n$ is far from  $f$ and it is better to take a smaller $\gamma$ (for instance $\gamma=0.5$).
The Gaussian density ({\bf G}) exhibits both behaviors. For large $n$ ($n\geq 1024$), 
there is a plateau phenomenon around $\gamma=1$. But for smaller $n$, the oracle $\bar{f}_n$ is not accurate enough and taking $\gamma=0.5$ is better.
Note finally that the choice $\gamma=1$, leading to our practical method, namely $\tilde{f}_n^{Prac}$, is the more robust with respect to both situations.

\section{The curse of support from a minimax point of view}

The goal of this section is to derive the minimax rates on the whole class of Besov spaces.  The subsequent results will constitute generalizations of the results derived in \citet{jll} who pointed out minimax rates for density estimation on the class of H\"older spaces. For this purpose, we  consider the theoretical procedure $\tilde f_{n,\ga}$ defined with the choice $c'=-c$ (see {\bf{Step 0}}) where
the real number  $c$ is  chosen later. 
In some situations, it will be necessary to strengthen our assumptions. More precisely, sometimes, we assume that $f$ is bounded. So, for any $R>0$, we consider the following set of functions:
$${\cal  L}_{2,\infty}(R)=\left\{f \mbox{ is a density such that } \norm{f}_2\leq
R \mbox{ and } \norm{f}_{\infty}\leq R\right\}.$$
The Besov balls we consider are classical (see Appendix A for a definition with respect to the biorthogonal wavelet basis) and  denoted ${\cal B}^{\al}_{p,q}(R)$. Let us just point out that no restriction is made on the support of $f$ when $f$ belongs to ${\cal B}^{\al}_{p,q}(R)$: this support is potentially the whole real line.
Now, let us state the upper bound of the $\L_2$-risk of $\tilde f_{n,\ga}$.
\begin{Th}\label{minimaxnoncompact}
Let $R,R'>0$, $1\leq p,q\leq \infty$ and
$\al\in\R$ such that $\max\left(0,\frac{1}{p}-\frac{1}{2}\right)<\al<r+1$, where we recall that $r$ ($r>0$) denotes the wavelet smoothness parameter introduced in
Appendix A.
Let $c\geq 1$ such that
\begin{equation}\label{sural}
\al\left(1-\frac1{c(1+2\alpha)}\right)\geq \frac1p -\frac12
\end{equation}
and $\ga >c$. Then, there exists a constant $C$ depending on $R'$, $\ga$, $c$, on the parameters of the Besov ball and on the choice of the biorthogonal wavelet basis such that for any $n$,
\begin{itemize}
 \item[-] if $p\leq 2$,
\begin{equation}\label{risqueBesov1}
\sup_{f\in       {\cal       B}^{\al}_{p,q}(R)\cap{\cal L}_{2,\infty}(R')}\E\left[\norm{\tilde{f}_{n,\ga}-f}^2\right]\leq
C\left(\frac{\ln n}{n}\right)^{\frac{2\al}{2\al+1}},
\end{equation}
\item [-] if $p> 2$,
\begin{equation}\label{risqueBesov2}
\sup_{f\in {\cal B}^{\al}_{p,q}(R)\cap \L_2(R')}\E\left[\norm{\tilde{f}_{n,\ga}-f}^2\right]\leq
C\left(\frac{\ln n}{n}\right)^{\frac{\al}{\al+1-\frac{1}{p}}}.
\end{equation}
\end{itemize}
\end{Th}

First, let us briefly comment assumptions of these results. When $p> 2$, (\ref{sural}) is satisfied and the result is true for any $c\geq 1$ and $0<\al<r+1$. In addition, we do not need to restrict ourselves to the set of bounded functions. When $p\leq 2$, the result is true as soon as $c$ is large enough to satisfy (\ref{sural}) and we establish (\ref{risqueBesov1}) only for bounded functions. Actually, this assumption is in some sense unavoidable as proved in Section 6.4 of \citet{birdensity}.

Furthermore, note that if we additionally assume that $f$ is bounded with a bounded support (say $[0,1]$) then $\E\left[\norm{\tilde{f}_{n,\ga}-f}^2\right]$ is always upper bounded by a constant times $\left(\ln n /n\right)^{\frac{2\al}{2\al+1}}$ whatever $p$ is, since, in this case,  the assumption $f\in {\cal   B}^{\al}_{p,\infty}(R)$ implies $f\in{\cal
B}^\al_{2,\infty}(\tilde R)$ for $\tilde R$ large enough and $p>2$.

Now, combining upper bounds (\ref{risqueBesov1}) and (\ref{risqueBesov2}), under assumptions of Theorem \ref{minimaxnoncompact}, we point out the following rate for our procedure when $f$ is bounded but without any assumption on the support:
$$
\sup_{f\in       {\cal       B}^{\al}_{p,q}(R)\cap{\cal L}_{2,\infty}(R')}\E\left[\norm{\tilde{f}_{n}-f}^2\right]\leq
C\left(\frac{\ln n}{n}\right)^{\frac{\al}{\al+\frac{1}{2}+\left(\frac{1}{2}-\frac{1}{p}\right)_+}}.
$$
The following result derives lower bounds of the minimax risk showing that this rate is the optimal rate up to a logarithmic term. So, the next result establishes the optimality properties of $\tilde{f}_{n,\ga}$ under the minimax approach.
\begin{Th}\label{lowerBesov} Let $R,R'>0$, $1\leq p,q\leq \infty$ and
$\al\in\R$ such that $\max\left(0,\frac{1}{p}-\frac{1}{2}\right)<\al<r+1$. Then, there exists a positive constant $\tilde C$ depending on $R'$ and  on the parameters of the Besov ball such that
$$\lim\inf_{n\to +\infty}n^{\frac{\al}{\al+\frac{1}{2}+\left(\frac{1}{2}-\frac{1}{p}\right)_+}}\inf_{\hat f}\sup_{f\in       {\cal       B}^{\al}_{p,q}(R)\cap{\cal L}_{2,\infty}(R')}\E\left[\norm{\hat{f}-f}^2\right]\geq
\tilde C,$$
where the infimum is taken over all the possible density estimators $\hat{f}$.

\noindent Furthermore, let $c$, $p^*\geq 1$ and $\al^*>0$ such that
\begin{equation}\label{adapt*}
\al^*\left(1-\frac1{c(1+2\alpha^*)}\right)\geq \frac{1}{p^*} -\frac12.
\end{equation}
 Then our procedure, $\tilde{f}_{n,\ga}$, constructed with this precise choice of $c$ and $\ga>c$, is adaptive minimax up to a logarithmic term on $$\left\{{\cal
B}^{\al}_{p,q}(R)\cap{\cal  L}_{2,\infty}(R'):\quad   \al^*\leq\al<  r+1,  \
p^*\leq p\leq +\infty,\ 1\leq
q\leq\infty\right\}.$$
\end{Th}
When $p\leq 2$, the lower bound for the minimax risk corresponds to
the classical minimax rate  for estimating a compactly supported density (see \citet{djkp}). In addition, the procedure $\tilde f_{n,\ga}$ achieves this minimax rate up to a logarithmic term. When $p>2$, the  risk deteriorates, if no assumption on the support is made, whereas it remains the same when we add the bounded support assumption. Note that when $p=\infty$, the exponent becomes
$\al/(1+\al)$: this rate was also derived in \citet{jll} for estimation on balls of ${\cal B}^{\al}_{\infty,\infty}$.

To summarize, we gather in Table \ref{tableau} the lower bounds for the minimax rates  obtained for  each situation. Those bounds are adaptively achieved by  our estimator with respect to $p$, $\al$ and the compactness of the support, up to a logarithmic term. If the logarithmic term is known to be unnecessary in the bounded support case, the question remains open in the other case.
\begin{table}[htb]
\begin{center}
\begin{tabular}{|l||c|c|}
\hline
&$1\leq p\leq 2$&$2\leq p\leq \infty$\\
\hline\hline
&&\\
compact support&$n^{-\frac{2\al}{2\al+1}}$&$n^{-\frac{2\al}{2\al+1}}$\\
\hline
&&\\
non compact support&$n^{-\frac{2\al}{2\al+1}}$&$n^{-\frac{\al}{\al+1-\frac{1}{p}}}$\\
\hline
\end{tabular}
\caption{Minimax rates on ${\cal
B}^{\al}_{p,q}\cap{\cal  L}_{2,\infty}$(up to a logarithmic term)  with $1\leq p,q\leq \infty$, $\al>\max\left(0,\frac{1}{p}-\frac{1}{2}\right)$  under the $\|\cdot\|_2^2$-loss.}\label{tableau}\end{center}
\end{table}

Our results show the role played by the support of the functions to be estimated on minimax rates. As already observed, when $p\leq 2$, the support has no influence since the rate exponent remains unchanged whatever the size of the support (finite or not). Roughly speaking, it means that it is not harder
to estimate bounded non-compactly supported  functions than bounded compactly supported functions from the minimax point of view. It is not the case when $p>2$. Actually, we note an elbow phenomenon at $p=2$ and the rate deteriorates when $p$ increases: this illustrates the curse of support from a minimax point of view. Let us give an interpretation of this observation. \citet{joh} showed that when $p<2$, Besov spaces ${\cal B}^{\al}_{p,q}$ model sparse signals where at each level, a very few number of the wavelet coefficients are non-negligible. But these coefficients can be very large. When $p>2$, ${\cal B}^{\al}_{p,q}$-spaces typically model dense signals where the wavelet coefficients are not large but most of them can be non-negligible. This explains why the size of the support plays a role on minimax rates when $p>2$: when the support is larger, the number of wavelet coefficients to be estimated increases dramatically.

Since arguments for proving Theorems \ref{minimaxnoncompact} and \ref{lowerBesov} are similar to the arguments used in \citet{Poisson_minimax}, proofs are omitted. We just mention that these results are derived from the oracle inequality established in Theorem \ref{inegoraclelavraie}.


\section{The curse of support from a practical point of view}
Now let us turn to a practical point of view. Is there a curse of support too? First we provide a simulation study illustrating the distortion of the most classic support dependent estimators when the support or the tail is increasing. Next we provide an application of our method to famous real data sets, namely the Suicide data and the Old Faithful geyser data.

\subsection{Simulations}
We compare our method to representative methods of each main trend in density estimation, namely kernel,  binning plus thresholding and model selection.
 The considered methods are the following. The first one is  the kernel method, denoted {\bf K}, consisting in a basic cross-validation choice of a global bandwidth with a Gaussian kernel.
The second method requires a complex preprocessing of the data based on binning. Observations $X_1,\ldots,X_n$ are first rescaled and centered by an affine transformation denoted $T$ such that $T(X_1),\ldots,T(X_n)$ lie in $[0,1]$. We denote $f_T$ the density of the data induced by the transformation $T$. We divide the interval $[0,1]$ into $2^{b_n}$ small intervals of size $2^{-b_n}$, where $b_n$ is an integer, and count the number of observations in each interval. We apply the root transform due to \citet{cai} and the universal hard individual thresholding rule on the coefficients computed with the DWT Coiflet-basis filter. We finally apply the unroot transform to obtain an estimate of $f_T$ and the final estimate of the density is obtained by applying $T^{-1}$ combined with a spline interpolation. This method is denoted {\bf RU}.
The last method is also support dependent. After rescaling as previously the data, we estimate $f_T$ by the algorithm of \citet{wn}. It consists in a complex  selection of a grid and of polynomials on that grid that minimizes a penalized  loglikelihood criterion.  The final estimate of the density is obtained by applying $T^{-1}$. This method is denoted {\bf WN}.

\noindent
Our practical method is implemented in the Haar basis (method {\bf H}) and in the Spline basis (method {\bf S})(see Figure \ref{Spline} in Appendix A for a complete description of this basis). Moreover  we have also implemented the choice $\gamma=0.5$ of (\ref{defthreshpracavecgamma}) in the Spline basis (see Section \ref{method}). We denote this method {\bf S*}.\\
  The thresholding rule proposed in \citet{jll} has also been considered. For their prescribed  practical choice of the tuning parameters and the Spline basis, the numerical performance is similar to those of method {\bf S}. 
Since thresholding is not performed for the coarsest level, the approximation term of the reconstruction is based on many non zero negligible coefficients for heavy-tailed signals: this  leads to obvious numerical difficulties without significant impact on the risk. So, numerical results of the thresholding rule proposed in \citet{jll} are not given in the sequel.

We generate $n$-samples of two kinds of densities $f$, with $n=1024$. Both signals are supported by the whole real line. We compute for each estimator $\hat{f}$ the 
ISE, i.e. $\int_\R (f-\hat{f})^2$ which is approximated by a trapezoidal method on a finite interval, adequately chosen so that the remaining term is negligible with respect to the ISE.

The first signal, $g_d$, consists in a mixture of two standard Gaussian densities:
$$g_d=\frac{1}{2}\, \mathcal{N}(0,1)+\frac12\, \mathcal{N}(d,1),$$
where $\mathcal{N}(\mu,\si)$ represents the density of a Gaussian variable with mean $\mu$ and standard deviation $\si$.
The parameter $d$ varies in $\{10,30,50,70\}$ so that we can see the curse of support on the quality of estimation.
\begin{figure}[h!]
\begin{center}
\hspace{-0.5cm} \includegraphics[width=0.8\linewidth,angle=0]{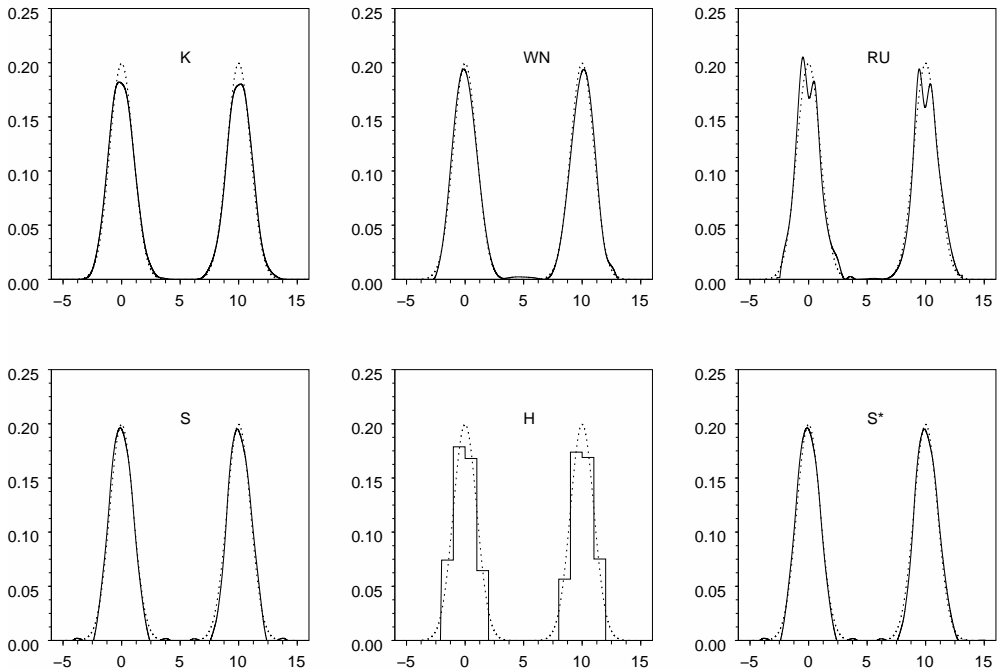}
\end{center}
\caption{Reconstruction of $g_d$ (true: dotted line, estimate: solid line) for the 6 different methods for $d=10$}
\label{ecarte10}
\end{figure}

\begin{figure}[h!]
\begin{center}
\hspace{-0.5cm} \includegraphics[width=0.78\linewidth,angle=0]{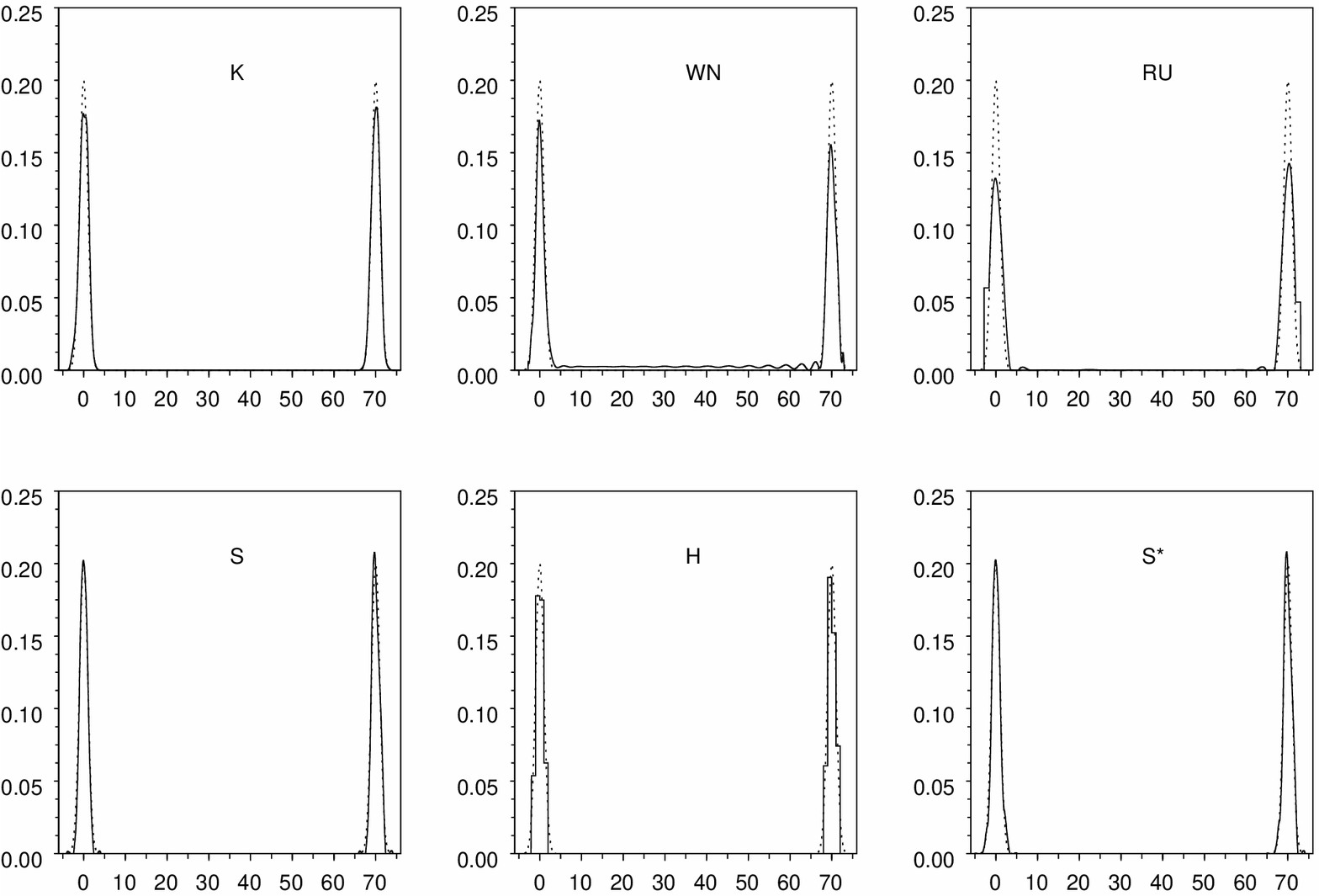}
\end{center}
\caption{Reconstruction of $g_d$ (true: dotted line, estimate: solid line) for the 6 different methods for $d=70$}
\label{ecarte70}
\end{figure}

\begin{figure}[h!]
\begin{center}
\hspace{-0.5cm} \includegraphics[width=0.48\linewidth,angle=-90]{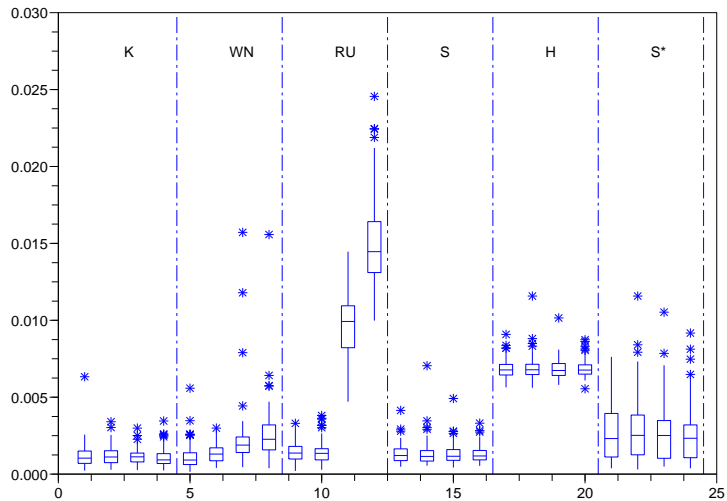}
\end{center}
\caption{Boxplot of the ISE for $g_d$ over 100 simulations for the 6 methods and the 4 different values of $d$. A column, delimited by dashed lines, corresponds to one method (respectively {\bf K, WN, RU, S, H, S*}). Inside this column, from left to right, one can find for the same method  the boxplot of the ISE for respectively $d=10, 30, 50$ and $70.$ }\label{boxecart}
\end{figure}

Figure \ref{ecarte10} shows the reconstructions for $d=10$ and Figure \ref{ecarte70} for $d=70$. In the sequel,  the method {\bf RU} is implemented  with $b_n=5$, which is the best choice for the reconstruction with $d=10$. All the methods give satisfying results for $d=10$. When $d$ is large, the rescaling and binning preprocessing leads to a poor regression signal which makes the regression thresholding rules non convenient, as illustrated by the method {\bf RU} with $d=70$. Reconstructions for  {\bf K, WN, S} and {\bf S*} seem satisfying but a study of the ISE of each method (see Figure \ref{boxecart}) reveals that both support dependent methods ({\bf RU} and {\bf WN}) have a risk that increases with $d$. On the contrary, methods {\bf K} and {\bf S} are the best ones and more interestingly their performance does not vary with $d$. This robustness is also true for {\bf H} and {\bf S*}. {\bf S*} is a bit undersmoothing: this was already noticed in Figure \ref{tracegamma} ({\bf G}) and this explains the variability of its ISE. Finally note that, for large $d$, {\bf H} is even better than {\bf RU} despite the inappropriate choice of the Haar basis.

The other signal, $h_k$, is both heavy-tailed and irregular. It consists in a mixture of 4 Gaussian densities and one Student density:
$$h_k=0.45\, T(k)+ 0.15\,\mathcal{N}(-1,0.05) + 0.1\,\mathcal{N}(-0.7,0.005)+ 0.25\,\mathcal{N}(1,0.025) + 0.15\,\mathcal{N}(2,0.05),$$
where $T(k)$ denotes the density of a Student variable with $k$ degrees of freedom.
The parameter $k$ varies in $\{2,4,8,16\}$. The smaller $k$, the heavier the tail is and this without changing the shape of the main part that has to be estimated. Figure \ref{herisson}  shows the reconstruction for $k=2$. Clearly {\bf RU} does not detect the local spikes at all. Indeed the maximal observation may be equal to $1000$ and the binning effect is disastrous. The kernel method {\bf K} clearly suffers from a lack of spatial adaptivity, as expected. The four remaining methods seem satisfying. In particular for this very irregular signal it is not clear that the Haar basis is a bad choice. Note however that to represent reconstructions, we   have
focused on the area where the spikes are located. In particular the support dependent method {\bf WN} is non zero on a very large interval, which tends to deteriorate its ISE. Indeed, Figure \ref{boxherisson} shows that the ISE of the support dependent methods ({\bf RU, WN}) increases when the tail becomes heavier, whereas the other methods have remarkable stable ISE. Methods {\bf S} and {\bf H} are more robust and better than {\bf WN} for $k=2$. The ISE may be improved for this irregular signal by taking $\gamma=0.5$ (see method {\bf S*}) as  already noticed  in Section 2 for irregular signals.

\begin{figure}[h!]
\begin{center}
\hspace{-0.5cm} \includegraphics[width=0.78\linewidth,angle=0]{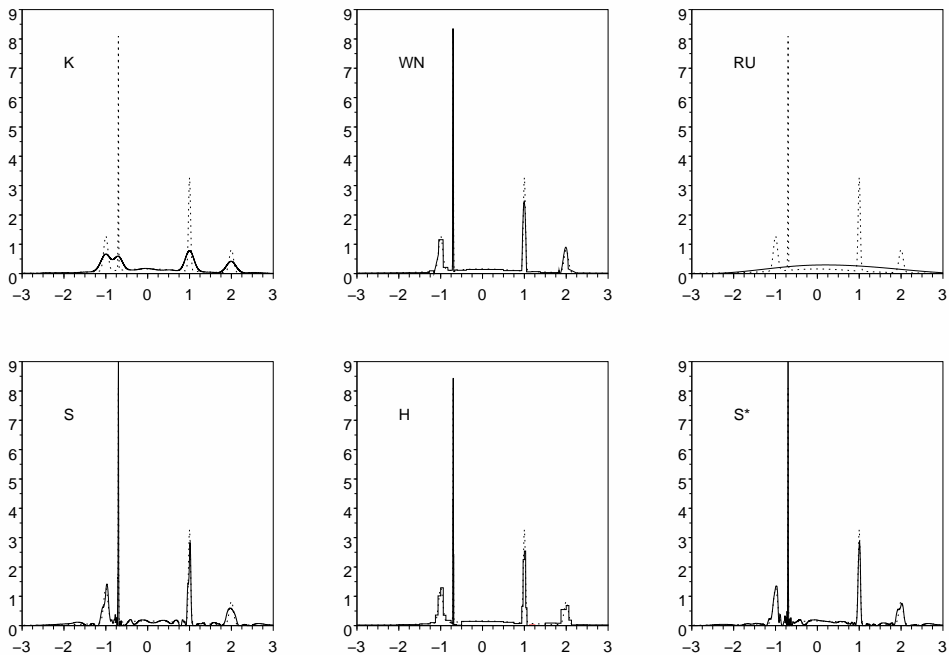}
\end{center}
\caption{Reconstruction of $h_k$ (true: dotted line, estimate: solid line) for the 6 different methods for $k=2$}\label{herisson}
\end{figure}

\begin{figure}[h!]
\begin{center}
\hspace{-0.5cm} \includegraphics[width=0.48\linewidth,angle=-90]{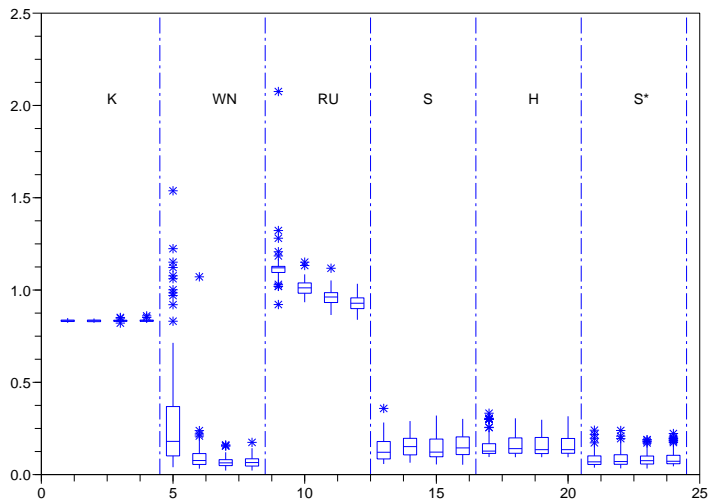}
\end{center}
\caption{Boxplot of the ISE for $h_k$ over 100 simulations for the 6 methods and the 4 different values of $k$. A column, delimited by dashed lines, corresponds to one method (respectively {\bf K, WN, RU, S, H, S*}). Inside this column, from left to right, one can find for the same method  the boxplot of the ISE for respectively $k=2, 4, 8$ and $16.$ }\label{boxherisson}
\end{figure}
\subsection{On  real data}
To illustrate and evaluate our procedure on real data, we consider two real data sets named, respectively in our study, ``Old Faithful geyser'' and ``Suicide''. The ``Old Faithful geyser'' data are the duration, in minutes, of $107$ eruptions of Old Faithful geyser located in Yellowstone National Park, USA; they are taken from \citet{Weisberg}. The ``Suicide'' data set is related to the study of suicide risks. Indeed, each of the $86$ observations corresponds to the number of days a patient, considered as control in the study, undergoes psychiatric treatment.  The data are available in \citet{copas}.
In both cases, we consider that we have a sample of $n$ real observations $X_1, \ldots, X_n$ and we want to estimate the underlying density $f$. We mention that in the first situation, all the observations are continuous whereas, in the second one, the observations are discrete. These data are well known and have been widely studied elsewhere. This allows to compare our procedure with other methods.

\begin{figure}[h!]
\begin{center}
\hspace{-0.5cm} \includegraphics[width=0.5\linewidth,angle=-0]{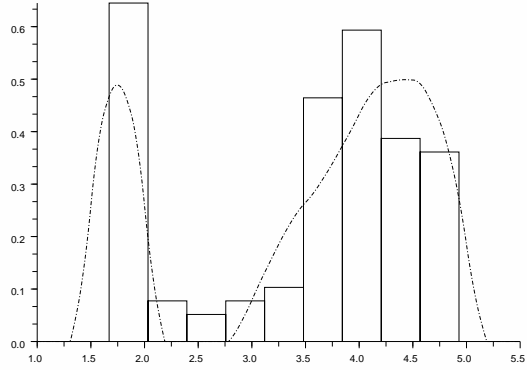}
\end{center}
\caption{Histogram (solid line) and reconstruction via $\tilde{f}_{n}^{Prac}$ (dashed line) for the "Old Faithful geyser" data set}\label{geyser}
\end{figure}

\noindent
To estimate the function $f$, we apply  $\tilde{f}_n^{Prac}$, with the Spline basis (see Figure \ref{Spline} in Appendix A) and $j_0=7$.
We plot, on the same graph
the resulting estimate and the histogram of the data. Figures \ref{geyser} and \ref{suicide} represent, respectively, the results for the ``Old Faithful geyser'' set and for the ``Suicide'' one.
Note that concerning the "Suicide" data set, there exists a problem of "scale": if we look at the associated histogram, the scale of the data seems to be approximately equal to 250, and not 1. So we divide the data by 250 before proceeding to the estimation.

\noindent
Respectively two or three peaks are detected providing multimodal reconstructions.  So, in comparison with the ones performed in \citet{silver} and \citet{SainScott}, our estimate  detects significant events and not artefacts. More interestingly, both estimates equal zero on an interval located between the last two peaks.
This cannot occur with the Gaussian kernel estimate mentioned previously. Of course, this has a strong impact for practical purposes, so this point is crucial. This tends to show that the proposed procedure is relevant for real data, even for relatively small sample~size.

\begin{figure}[h!]
\begin{center}
\hspace{-0.5cm} \includegraphics[width=0.5\linewidth,angle=-0]{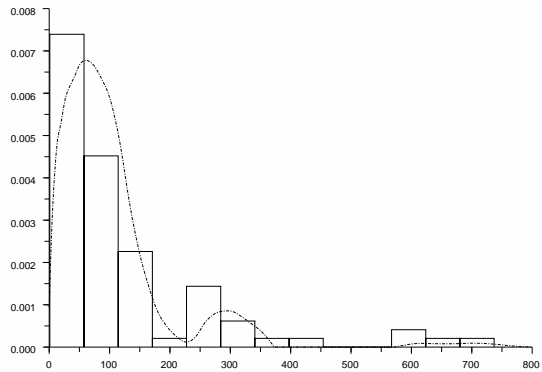}
\end{center}
\caption{Histogram (solid line) and reconstruction via $\tilde{f}_{n}^{Prac}$ (dashed line) for the "Suicide" data set}\label{suicide}
\end{figure}

\appendix
\section{Analytical tools}
All along this paper,  we have considered a particular
class of wavelet bases that are described now.
We set
$$\phi=\indic_{[0,1]}.$$ For any $r> 0$, we can claim that there exist three functions
$\psi$, $\tilde \phi$ and $\tilde\psi$ with the following properties:
\begin{enumerate}
\item $\tilde\phi$ and $\tilde\psi$ are compactly supported,
\item  $\tilde\phi$  and $\tilde\psi$  belong  to  $C^{r+1}$, where  $C^{r+1}$
denotes the H\"older space of order $r+1$,
\item $ \psi$ is compactly supported and is a piecewise constant function,
\item $\psi$ is orthogonal to polynomials of degree no larger than $r$,
\item $\{(\phi_{k},\psi_{jk})_{j\geq
0,k\in\Z},(\tilde\phi_{k},\tilde\psi_{jk})_{j\geq     0,k\in\Z}\}$    is    a
biorthogonal family: for any $j,j'\geq 0,$ for any $k,k',$
$$\int_\R\psi_{jk}(x)\tilde\phi_{k'}(x)dx=\int_\R\phi_{k}(x)\tilde\psi_{j'k'}(x)dx=0,$$
$$\int_\R\phi_{k}(x)\tilde\phi_{k'}(x)dx=1_{k=k'},\quad
\int_\R\psi_{jk}(x)\tilde\psi_{j'k'}(x)dx=1_{j=j',k=k'},$$
where for any $x\in\R$,
$$\phi_{k}(x)=\phi(x-k), \quad \psi_{jk}(x)=2^{\frac{j}{2}}\psi(2^jx-k)$$
and
$$\tilde\phi_{k}(x)=\tilde\phi(x-k), \quad \tilde\psi_{jk}(x)=2^{\frac{j}{2}}\tilde\psi(2^jx-k).$$
\end{enumerate}
This implies the following wavelet decomposition of $f\in \L_2(\R)$:
$$
f=\sum_{k\in\Z}\alpha_k\tilde\phi_k+\sum_{j\geq
0}\sum_{k\in\Z}\beta_{jk}\tilde\psi_{jk},
$$
where for any  $j\geq 0$ and any $k\in\Z$,
$$\alpha_k=\int_\R          f(x)\phi_k(x)dx,\quad          \beta_{jk}=\int_\R
f(x)\psi_{jk}(x)dx.$$
Such biorthogonal wavelet bases have been built by \citet{cdf} as a special
case  of spline  systems (see  also the  elegant  equivalent construction  of \citet{don}  from boxcar functions).
The  Haar basis can be  viewed as a
particular biorthogonal wavelet basis, by setting  $\tilde\phi=\phi$ and
$\tilde\psi=\psi=\indic_{[0,\frac{1}{2})}-\indic_{[\frac{1}{2},1]}$, with $r=0$ (even if Property 2
is not satisfied with such a choice). The Haar basis is
an orthonormal basis, which is not true for general biorthogonal wavelet
bases. However, we have the frame property: if we denote
$$\varPhi=\{\phi,\psi,\tilde\phi,\tilde\psi\}$$
there exist two constants $c_1(\varPhi)$ and $c_2(\varPhi)$ only depending on $\varPhi$ such that
$$c_1(\varPhi)\left(\sum_{k\in\Z}\alpha_k^2+\sum_{j\geq
0}\sum_{k\in\Z}\beta_{jk}^2\right)\leq                         \|f\|_{2}^2\leq
c_2(\varPhi)\left(\sum_{k\in\Z}\alpha_k^2+\sum_{j\geq
0}\sum_{k\in\Z}\beta_{jk}^2\right).$$
For    instance,    when    the    Haar    basis    is    considered,
$c_1(\varPhi)=c_2(\varPhi)=1$.

\begin{figure}[tpb]
\begin{center}
\hspace{-0.5cm} \includegraphics[width=0.5\linewidth,angle=-0]{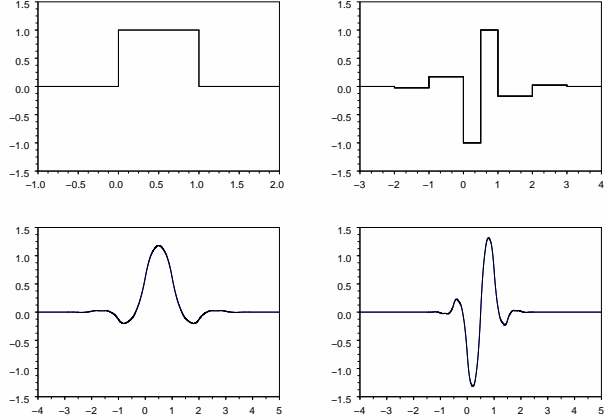}
\end{center}
\caption{Biorthogonal wavelet basis with $r=0.272$ that  is used in the Simulation study. First line,  $\phi$ (left) and $\psi$ (right), second line $\tilde\phi$ (left) and $\tilde\psi$ (right).  }\label{Spline}
\end{figure}

We emphasize the important feature of such bases: the functions $\psi_{jk}$ are piecewise constant functions. For instance, Figure \ref{Spline} shows an example which is the one that has been implemented for numerical studies. This allows to compute easily wavelet coefficients without using the discrete wavelet transform. In addition, there exists a constant $\mu_{\psi}>0$ such that
$$
\inf_{x\in[0,1]}|\phi(x)|\geq 1,\quad\inf_{x\in \Supp(\psi)}|\psi(x)|\geq\mu_{\psi},
$$
where $\Supp(\psi)=\{x\in\R:\quad \psi(x)\not=0\}.$

This technical feature will be used through the proofs of our results.
To shorten mathematical
expressions, we have previously set for any $k\in\Z$, $\tilde\p_{-1k}=\tilde\phi_k$, $\p_{-1k}=\phi_k$ and $\be_{-1k}=\al_k$.

Now, let us give some properties of Besov
spaces.  Besov spaces, denoted ${\cal B}^{\al}_{p,q}$, are classically
defined by using modulus of continuity (see \citet{dl} and \citet{hardle}).  We just  recall here the sequential characterization of Besov
spaces by using the biorthogonal wavelet basis (for further details, see \citet{dj}).

Let $1\leq  p,q\leq \infty$ and $0<\al<r+1$, the  ${\cal B}^\al_{p,q}$-norm of
$f$ is equivalent to the norm
$$
\norm{f}_{\al,p,q}=\left\{
\begin{array}{ll}
\norm{(\al_{k})_k}_{\ell_p}+\left[\sum_{j\geq 0}2^{jq(\al+\frac{1}{2}-\frac{1}{p})}\norm{(\be_{j,k})_{k}}_{\ell_p}^{q}\right]^{1/q}&\mbox{ if } q<\infty,\\
\norm{(\al_{k})_k}_{\ell_p}+\sup_{j\geq 0}2^{j(\al+\frac{1}{2}-\frac{1}{p})}\norm{(\be_{j,k})_{k}}_{\ell_p}&\mbox{ if } q=\infty.
\end{array}
\right.
$$
We  use this norm to define Besov balls with radius $R$
$${\cal B}^{\al}_{p,q}(R)=\{f \in \L_2(\R):\quad\norm{f}_{\al,p,q}\leq R\}.$$
  For any $R>0$, if
$0<\al'\leq\al<r+1$, $1\leq  p\leq p'\leq\infty$ and  $1\leq q\leq q'\leq\infty$,
we obviously have
$$
\mathcal{B}^\al_{p,q}(R)\subset\mathcal{B}^\al_{p,q'}(R),\quad
\mathcal{B}^\al_{p,q}(R)\subset\mathcal{B}^{\al'}_{p,q}(R).$$
Moreover
$$\mathcal{B}^{\al}_{p,q}(R)\subset\mathcal{B}^{\al'}_{p',q}(R)
\mbox{ if } \al-\frac{1}{p}\geq \al'-\frac{1}{p'}.$$
The class of Besov  spaces  provides a useful tool
to classify  wavelet decomposed  signals with respect to their  regularity and
sparsity properties (see \citet{joh}). Roughly speaking, regularity increases when $\al$ increases whereas
sparsity increases when $p$ decreases.

\section{Proofs}

\subsection{Proof of Theorem \ref{inegoraclelavraie}}
Because of the frame property of the biorthogonal wavelet basis, it is easy to see that
\begin{equation}\label{equnormes}
c_1(\varPhi)\norm{\tilde\be-\be}_{\ell_2}^2\leq \|\tilde f_{n,\ga}-f\|_{2}^2\leq
c_2(\varPhi)\norm{\tilde\be-\be}_{\ell_2}^2,
\end{equation}
where $\tilde\be$ denotes the sequence of thresholded coefficients $(\tb_\la \indic_{(j,k)\in\Gamma_n})_{(j,k)\in\Lambda}$ and $\be$ denotes the true coefficients $(\be_\la)_{(j,k)\in\Lambda}$.
Consequently, it is sufficient to restrict ourselves to the study of the $\norm{\tilde\be-\be}_{\ell_2}^2$.

Consequently the proof of Theorem \ref{inegoraclelavraie} relies on the following result (see Theorem 7 of Section 4.1 in \citet{Poisson_minimax}).

\begin{Th}\label{inegmodelsel} Let $\La$ be a set of indices. To estimate a countable family $\be=(\be_{\lambda})_{\lambda \in \La}$ such that $\|\be\|_{\ell_2}<\infty$, we assume that
 a family of coefficient estimators $(\hb_{\lambda})_{\lambda \in \Ga}$, where $\Ga$ is
a known deterministic subset of $\La$, and a family of possibly
random thresholds $(\eta_{\lambda})_{\lambda \in \Ga}$ are available and we consider
the thresholding rule $\tb=(\hb_{\lambda}\indic_{|\hb_{\lambda}|\geq
  \eta_{\lambda}}\indic_{\lambda \in \Ga})_{\lambda \in \La}$.
Let $\e>0$ be fixed.
Assume  that there exist  a deterministic  family $(F_{\lambda})_{\lambda \in\Ga}$ and
three constants $\kappa\in [0,1[$,  $\omega\in [0,1]$ and $\mu>0$ (that
may depend on $\e$ but not on $\lambda$) with the following properties.
\begin{itemize}
\item[(A1)] For   all  $\lambda \in\Ga$,  $$\P(|\hb_{\lambda}-\be_{\lambda}|>\kappa\eta_{\lambda})\leq\omega.$$
\item[(A2)] There exist $1<p,q<\infty$ with $\frac{1}{p}+\frac{1}{q}=1$ and a constant $R>0$ such that for all $\lambda \in\Ga$,
  $$\left(\E(|\hb_{\lambda }-\be_{\lambda}|^{2p})\right)^{\frac{1}{p}}\leq R \max(F_{\lambda},
  F_{\lambda}^{\frac{1}{p}}\e^{\frac{1}{q}}).$$
\item[(A3)] There exists  a constant $\theta$ such that for all $\lambda \in\Ga$ satisfying $F_{\lambda}<\theta \e$
$$\P(|\hb_{\lambda}-\be_{\lambda}|>\kappa\eta_{\lambda}\ , |\hb_{\lambda}|>\eta_{\lambda})\leq F_{\lambda}\mu.$$
\end{itemize}
Then the estimator $\tb$ satisfies
$$\frac{1-\kappa^2}{1+\kappa^2}\E\|\tb-\be\|_{\ell_2}^2\leq \E\inf_{m\subset
  \Ga}\left\{\frac{1+\kappa^2}{1-\kappa^2}\sum_{\lambda \not \in
  m}\be_{\lambda}^2+\frac{1-\kappa^2}{\kappa^2}\sum_{\lambda \in
  m}(\hb_{\lambda}-\be_{\lambda})^2+\sum_{{\lambda}\in m}\eta_{\lambda}^2\right\}+LD\sum_{\lambda \in\Ga}F_{\lambda}
$$
with $$LD=\frac{R}{\kappa^2}\left(\left(1+\theta^{-1/q}\right)\omega^{1/q}+(1+\theta^{1/q})\e^{1/q}\mu^{1/q}\right).$$
\end{Th}
To prove Theorem \ref{inegoraclelavraie},  we  use Theorem \ref{inegmodelsel} with $\lambda=(j,k)$,  $\hb_\lambda=\hb_\la$  defined in  (\ref{defest1}),   $\eta_\la=\eta_{\la,\ga}$  defined   in
(\ref{defthresh})  and
$$\Ga=\Ga_n=\left\{(j,k)\in\Lambda:\  -1\leq  j\leq  j_0\right\}\mbox{
with } \
2^{j_0}\leq n^c(\ln n)^{c'}<2^{j_0+1} .$$
We set
$$F_\la=\int_{\Supp(\p_\la)}f(x)dx.$$
Hence we have:
\begin{equation}\label{Fla}
\sum_{(j,k)\in\Ga_n}F_\la=\sum_{-1\leq j\leq j_0}\sum_k\int_{x\in
\Supp(\p_{jk})} f(x)dx\leq\int f(x)dx\sum_{-1\leq j\leq
j_0}\sum_k\indic_{x\in \Supp(\p_{jk})} \leq (j_0+2)m_\p,
\end{equation}
where $m_\p$  is a finite constant depending  only on the  compactly  supported function
$\psi$.   Finally, $\sum_{(j,k)\in\Ga_n}F_\la$  is
bounded by $\log(n)$ up to a  constant that only depends on $c$,
$c'$ and the function $\psi$. Now, we give a fundamental lemma to
derive Assumption (A1) of  Theorem~\ref{inegmodelsel}.
\begin{Lemma}
For any $\gamma>1$ and any $\e'>0$ there exists a constant $M$ depending on $\e$ and $\gamma$ such that
$$\P\left(\sigma_{jk}^2\geq (1+\e')\widetilde{\si}^2_{jk}\right)\leq M n^{-\gamma}.$$
\end{Lemma}
\begin{proof}
We have:
\begin{eqnarray}\label{sn}
\widehat{\si}^2_{jk}&=& \frac{1}{2n(n-1)}\sum_{i\not =l}(\p_{jk}(X_i)-\p_{jk}(X_l))^2\nonumber\\
&=&\frac{1}{2n}\sum_{i=1}^n(\p_{jk}(X_i)-\be_{jk})^2+\frac{1}{2n}\sum_{l=1}^n(\p_{jk}(X_l)-\be_{jk})^2\nonumber\\
&&\hspace{1cm}-\frac{2}{n(n-1)}\sum_{i=2}^n\sum_{l=1}^{i-1}(\p_{jk}(X_i)-\be_{jk})(\p_{jk}(X_l)-\be_{jk})\nonumber\\
&=& s_n-\frac{2}{n(n-1)} u_n
\end{eqnarray}
with
$$s_n=\frac{1}{n}\sum_{i=1}^n(\p_{jk}(X_i)-\be_{jk})^2 \quad\mbox{ and }\quad u_n= \sum_{i=2}^n\sum_{l=1}^{i-1}(\p_{jk}(X_i)-\be_{jk})(\p_{jk}(X_l)-\be_{jk}).$$
Using the Bernstein inequality (see section 2.2.3 in \citet{stflourpascal})  applied to the variables $Y_i$ with $$Y_i=\frac{\si_{jk}^2-(\p_{jk}(X_i)-\be_{jk})^2}{n}\leq \frac{\si_{jk}^2}{n},$$ one obtains for any $u>0$,
$$\P\left(\si^2_{jk}\geq s_n +\sqrt{2 v_{jk} u}+\frac{\si^2_{jk} u}{3n}\right)\leq e^{-u}$$
with
$$v_{jk}=\frac{1}{n}\E\left[\left(\si_{jk}^2-(\p_{jk}(X_i)-\be_{jk})^2\right)^2\right].$$
We have
\begin{eqnarray*}
v_{jk}&=& \frac{1}{n}\left(\si_{jk}^4+\E\left[(\p_{jk}(X_i)-\beta_{jk})^4\right]-2\si_{jk}^2\E\left[(\p_{jk}(X_i)-\beta_{jk})^2\right]\right)\\
&=&\frac{1}{n}\left(\E\left[(\p_{jk}(X_i)-\beta_{jk})^4\right]-\si_{jk}^4\right)\\
&\leq&\frac{\si_{jk}^2}{n}\left(\norm{\p_{jk}}_{\infty}+|\beta_{jk}|\right)^2\\
&\leq& \frac{4\si_{jk}^2}{n}\norm{\p_{jk}}_\infty^2 .
\end{eqnarray*}
Finally
\begin{equation}
\label{concsn}
\P\left(\si^2_{jk}\geq s_n +2\norm{\p_{jk}}_\infty\si_{jk}\sqrt{\frac{2u}{n}}+\frac{\si^2_{jk} u}{3n}\right)\leq e^{-u}.
\end{equation}

\noindent
Now, we deal with the degenerate U-statistics $u_n$. We use Theorem 3.1 of \citet{hr} combined with the appropriate choice of constants derived by \citet{kr}: for any $u>0$ and any $\tau>0$,
\begin{equation}
\label{concustats}
\P\left(u_n\geq (1+\tau)C\sqrt{2u}+2Du+\frac{1+\tau}{3} F u +\left(\sqrt{2}(3+\tau^{-1})+\frac{2}{3}\right) B u^{3/2}+ \frac{3+\tau^{-1}}{3}Au^2\right)\leq 3 e^{-u}.
\end{equation}
Now we need to define and control the 5 quantities $A,B,C,D$ and $F$. For this purpose, let us set for any $x$ and $y$,
 $$g_{jk}(x,y)=(\p_{jk}(x)-\be_{jk})(\p_{jk}(y)-\be_{jk}).$$
We have:
$$A=\norm{g_{jk}}_\infty \leq 4 \norm{\p_{jk}}_\infty^2.$$
Furthermore,
$$C^2=\sum_{i=2}^n\sum_{l=1}^{i-1} \E(g^2_{jk}(X_i,X_l))=\frac{n(n-1)}{2}\si^4_{jk}.$$
The next term is
\begin{eqnarray*}
D&=&\sup_{\E\sum a_i^2(X_i)\leq 1, \ \E\sum b_l^2(X_l)\leq 1}\E\left(\sum_{i=2}^n\sum_{l=1}^{i-1} g_{jk}(X_i,X_l) a_i(X_i)b_l(X_l)\right)\\
&=&\sup_{\E\sum a_i^2(X_i)\leq 1,\ \E\sum b_l^2(X_l)\leq 1}\sum_{i=2}^n\sum_{l=1}^{i-1}\E\left((\p_{jk}(X_i)-\be_{jk})a_i(X_i)\right)\E\left((\p_{jk}(X_l)-\be_{jk})b_l(X_l)\right)\\
&\leq & \sup_{\E\sum a_i^2(X_i)\leq 1,\ \E\sum b_l^2(X_l)\leq 1}\sum_{i=2}^n\sum_{l=1}^{i-1} \sqrt{\si^2_{jk}\E( a_i^2(X_i))}\sqrt{\si^2_{jk} \E (b_l^2(X_l))}.
\end{eqnarray*}
So, we have
\begin{eqnarray*}
D
&\leq & \si^2_{jk}  \sup_{\E\sum a_i^2(X_i)\leq 1,\ \E\sum b_l^2(X_l)\leq 1} \sum_{i=2}^n \sqrt{\E(a_i^2(X_i))} \sqrt{\sum_{l=1}^{i-1}\E(b_l^2(X_l))}\sqrt{i-1}\\
&\leq &\si^2_{jk}  \sup_{\E\sum a_i^2(X_i)\leq 1} \sqrt{\sum_{i=2}^n \E(a_i^2(X_i))} \sqrt{\sum_{i=2}^n{(i-1)}}\\
&\leq & \si^2_{jk} \sqrt{\frac{n(n-1)}{2}}.
\end{eqnarray*}
Still using Theorem 3.1 of \citet{hr}, we have:
\begin{eqnarray*}
B^2&=&\sup_{t} \sum_{l=1}^{n-1} \E((\p_{jk}(t)-\be_{jk})^2(\p_{jk}(X_l)-\be_{jk})^2)\\
&\leq & 4 (n-1)\norm{\p_{jk}}_\infty^2 \si^2_{jk}\\
&\leq& 4 (n-1) \norm{\p_{jk}}_\infty^4
\end{eqnarray*}
Finally
\begin{eqnarray*}
F &=&\E\left(\sup_{i,t}\Bigg|\sum_{l=1}^{i-1}(\p_{jk}(t)-\be_{jk})(\p_{jk}(X_l)-\be_{jk})\Bigg|\right)\\
&\leq & 2\norm{\p_{jk}}_\infty \E\left(\sup_i \Bigg|\sum_{l=1}^{i-1}(\p_{jk}(X_l)-\be_{jk})\Bigg|\right).
\end{eqnarray*}
To control this term, we set
$$Z_i=\sum_{l=1}^{i-1} (\p_{jk}(X_l)-\be_{jk}).$$
Applying Lemma 1 of \citet{devlugo}, for any $s>0$, for any $i\leq n$,
$$\E\left(e^{sZ_i}\right)\leq e^{s^2 4 \norm{\p_{jk}}_\infty^2 (i-1)/8}\leq e^{s^2 (n-1)\norm{\p_{jk}}_\infty^2/2}.$$
Similarly,
$$\E\left(e^{-sZ_i}\right)\leq e^{s^2 (n-1)\norm{\p_{jk}}_\infty^2/2}.$$ Hence, by Lemma 2.2 of \citet{devlugo},
$$\E\left(\sup_i \Bigg|\sum_{l=1}^{i-1}(\p_{jk}(X_l)-\be_{jk})\Bigg|\right)\leq \norm{\p_{jk}}_\infty\sqrt{2(n-1)\ln(2n)}.$$
Hence
$$F\leq 2\sqrt{2}\norm{\p_{jk}}_\infty^2 \sqrt{(n-1)\ln(2n)}.$$
Now, for any $u>0$, let us set $$S(u)=2\norm{\p_{jk}}_\infty\si_{jk}\sqrt{2 \frac{u}{n}}+\frac{\si^2_{jk} u}{3n}$$
and $$U(u)=(1+\tau)C\sqrt{2u}+2Du+\frac{1+\tau}{3} F u +\left(\sqrt{2}(3+\tau^{-1})+\frac{2}{3}\right) B u^{3/2}+ \frac{3+\tau^{-1}}{3}Au^2.$$
Inequalities (\ref{concsn}) and (\ref{concustats}) give
\begin{eqnarray*}
\P\left(\si^2_{jk}\geq \widehat{\si}^2_{jk}+ S(u)+\frac{2}{n(n-1)}U(u)\right)&=&\P\left(\si^2_{jk}\geq s_n+S(u) + \frac{2}{n(n-1)}(U(u)-u_n)\right)\\
&\leq & \P\left(\si^2_{jk}\geq s_n+S(u)\right)+\P(u_n\geq U(u))\\
&\leq & 4 e^{-u}.
\end{eqnarray*}
Let us take $u=\gamma \ln n$ and $\tau=1$.
Then, there exist some constants $a$ and $b$ depending on $\gamma$ such that
$$S(u)+\frac{2}{n(n-1)}U(u)\leq 2\si_{jk}\norm{\p_{jk}}_\infty\sqrt{2\gamma\frac{\ln n}{n}}+a \si^2_{jk} \frac{\ln n}{n}+ b\norm{\p_{jk}}_\infty^2 \left(\frac{\ln n}{n}\right)^{3/2}.$$
So,
$$\P\left(\si^2_{jk}\geq \widehat{\si}^2_{jk}+2\si_{jk}\norm{\p_{jk}}_\infty\sqrt{2\gamma\frac{\ln n}{n}}+a \si^2_{jk} \frac{\ln n}{n}+ b\norm{\p_{jk}}_\infty^2 \left(\frac{\ln n}{n}\right)^{3/2}\right)\leq 4n^{-\gamma}$$
and
$$\P\left(\si^2_{jk}\left(1-a \frac{\ln n}{n}\right)- 2\si_{jk}\norm{\p_{jk}}_\infty\sqrt{2\gamma\frac{\ln n}{n}}-\widehat{\si}^2_{jk}- b\norm{\p_{jk}}_\infty^2 \left(\frac{\ln n}{n}\right)^{3/2}\geq 0\right)\leq 4n^{-\gamma}.$$
Now, we set
$$\theta_1=\left(1-a \frac{\ln n}{n}\right),\quad \theta_2=\norm{\p_{jk}}_\infty\sqrt{2\gamma\frac{\ln n}{n}}$$
and
$$\theta_3=\widehat{\si}^2_{jk}+ b\norm{\p_{jk}}_\infty^2 \left(\frac{\ln n}{n}\right)^{3/2}$$with $\theta_1,\theta_2,\theta_3 >0$ for $n$ large enough depending only on $\gamma$.
We study the polynomial
$$p(\si)=\theta_1 \si^2-2\theta_2\si-\theta_3.$$
Then, since $\si\geq 0$, $p(\si)\geq 0$ means that
$$\si \geq  \frac{1}{\theta_1}\left(\theta_2+\sqrt{\theta_2^2+\theta_1\theta_3}\right),$$
which is equivalent to
$$\si^2 \geq \frac{1}{\theta_1^2}\left(2\theta_2^2+\theta_1\theta_3+2\theta_2\sqrt{\theta_2^2+\theta_1\theta_3}\right).$$
Hence
$$\P\left(\si^2_{jk} \geq \frac{1}{\theta_1^2}\left(2\theta_2^2+\theta_1\theta_3+2\theta_2\sqrt{\theta_2^2+\theta_1\theta_3}\right)\right)\leq 4n^{-\gamma}.$$
So,
$$\P\left(\si^2_{jk} \geq \frac{\theta_3}{\theta_1}+\frac{2\theta_2\sqrt{\theta_3}}{\theta_1\sqrt{\theta_1}}+\frac{4\theta_2^2}{\theta_1^2}\right)\leq 4n^{-\gamma}.$$
 So, there exist  absolute constants $\delta$, $\eta,$ and $\tau'$  depending only on $\gamma$ so that for $n$ large enough,
\begin{small}
$$\P\left(\si^2_{jk} \geq \widehat{\si}^2_{jk}\left(1+\delta \frac{\ln n}{n}\right)+ \left(1+\eta\frac{\ln n}{n}\right)2\norm{\p_{jk}}_\infty\sqrt{2\gamma\widehat{\si}^2_{jk}\frac{\ln n}{n}}+8\gamma \norm{\p_{jk}}_\infty^2 \frac{\ln n}{n}\left(1+\tau'\left(\frac{\ln n}{n}\right)^{1/4}\right)\right)\leq 4 n^{-\gamma}.$$  \end{small}
Hence, with
$$\widetilde{\si}^2_{jk}=\widehat{\si}^2_{jk}+2\norm{\p_{jk}}_\infty\sqrt{2\gamma\widehat{\si}^2_{jk}\frac{\ln n}{n}}+8\gamma \norm{\p_{jk}}_\infty^2 \frac{\ln n}{n},$$
for all $\e>0$ there exists $M$ such that
$$\P(\sigma^2_{jk}\geq (1+\e')\widetilde{\si}^2_{jk})\leq M n^{-\gamma}.$$
\end{proof}
Let $\kappa<1$. Applying the previous lemma gives
\begin{eqnarray*}
 \P(|\hb_\la-\be_\la|>\kappa\eta_{\la,\gamma})
& \leq& \P\left(|\hb_\la-\be_\la|\geq \sqrt{2\kappa^2\ga
    {\widetilde{\si}}^2_{jk}\frac{\ln n}{n}}+\frac{2\kappa\gamma\ln n
    \norm{\p_\la}_\infty }{3 n}\right)\\
&\leq& \P\left(|\hb_\la-\be_\la|\geq\sqrt{2\kappa^2\ga{\widetilde{\si}}^2_{jk}\frac{\ln n}{n}}+\frac{2\kappa\gamma\ln n
    \norm{\p_\la}_\infty }{3 n} , \ \sigma^2_{jk}\geq (1+\e')\widetilde{\si}^2_{jk}\right)\\&&+\P\left(|\hb_\la-\be_\la|\geq\sqrt{2\kappa^2\ga{\widetilde{\si}}^2_{jk}\frac{\ln n}{n}}+\frac{2\kappa\gamma\ln n
    \norm{\p_\la}_\infty }{3 n}, \ \sigma^2_{jk}< (1+\e')\widetilde{\si}^2_{jk}\right)\\
&\leq &\P\left( \sigma^2_{jk}\geq (1+\e)\widetilde{\si}^2_{jk}\right)\\
&&\hspace{1cm}+ \P\left(|\hb_\la-\be_\la|\geq \sqrt{2\kappa^2\ga(1+\e')^{-1}\si^2_{jk}\frac{\ln n}{n}}+\frac{2\kappa\gamma\ln n \norm{\p_\la}_\infty }{3n}\right).
\end{eqnarray*}
Using again the Bernstein inequality, we have for any $u>0$,
$$\P\left(|\hb_\la-\be_\la|\geq \sqrt{\frac{2u\si^2_{jk}}{n}}+\frac{2u\norm{\p_\la}_\infty }{3n}\right)\leq 2e^{-u}.$$
So, with $\e'=1-\kappa$, there exists a constant $M_\kappa$ depending only on $\kappa$ and $\gamma$ such that
$$\P(|\hb_\la-\be_\la|>\kappa\eta_{\la,\gamma}) \leq M_\kappa n^{-\gamma\kappa^2/(2-\kappa)}.$$
So, for any value of $\kappa\in [0,1[$, Assumption (A1) is true with $\eta_{jk}=\eta_{jk,\gamma}$  if we take $\omega=M_\kappa n^{-\gamma\kappa^2/(2-\kappa)}$.

\noindent
Now, to prove (A2), we use the Rosenthal inequality. There exists a constant $C(p)$ only depending on $p$ such that
\begin{eqnarray*}
\E(|\hb_\la-\be_\la|^{2p})&=&\frac{1}{n^{2p}} \E\left[\left|\sum_{i=1}^n\left(\psi_{jk}(X_i)-\E(\psi_{jk}(X_i))\right)\right|^{2p}\right]\\
&\leq&\frac{C(p)}{n^{2p}}\left(\sum_{i=1}^n\E\left[\left|\psi_{jk}(X_i)-\E(\psi_{jk}(X_i))\right|^{2p}\right]+\left(\sum_{i=1}^n\Var(\psi_{jk}(X_i))\right)^p\right)\\
&\leq&\frac{C(p)}{n^{2p}}\left(\sum_{i=1}^n\left(2\norm{\psi_{jk}}_\infty\right)^{2p-2}\Var(\psi_{jk}(X_i))+\left(\sum_{i=1}^n\Var(\psi_{jk}(X_i))\right)^p\right)\\
&\leq&\frac{C(p)}{n^{2p}}\left(\left(2\norm{\psi_{jk}}_\infty\right)^{2p-2}n\si^2_{jk}+n^p\si^{2p}_{jk}\right)\\
&\leq&\frac{C(p)}{n^{2p}}\left(\left(2\norm{\psi_{jk}}_\infty\right)^{2p}nF_{jk}+n^p\norm{\psi_{jk}}_\infty^{2p}F_{jk}^p\right).
\end{eqnarray*}
Finally,
\begin{eqnarray*}
\left(\E(|\hb_\la-\be_\la|^{2p})\right)^{\frac{1}{p}}&\leq&\frac{4C(p)^{\frac{1}{p}}\norm{\psi_{jk}}_\infty^2}{n}\left(n^{1-p}F_{jk}+F_{jk}^p\right)^{\frac{1}{p}}\\
&\leq&\frac{4C(p)^{\frac{1}{p}}2^{j_0}\max(\norm{\phi}_\infty^2;\norm{\psi}_\infty^2)}{n}\left(n^{-\frac{1}{q}}F_{jk}^{\frac{1}{p}}+F_{jk}\right).
\end{eqnarray*}
So,    Assumption    (A2)    is    satisfied   with    $\e=\frac{1}{n}$    and
$$R=\frac{8C(p)^{\frac{1}{p}}2^{j_0}\max(\norm{\phi}_\infty^2;\norm{\psi}_\infty^2)}{n}.$$

\noindent
Finally, to prove Assumption (A3), we use the following lemma.
\begin{Lemma}
\label{nombredepoints} We set $$N_\la=\sum_{i=1}^n\indic_{\{X_i\in\Supp(\psi_{jk})\}} \quad\mbox{ and }\quad C'=\frac{14\gamma}{3}\geq\frac{14}{3}.$$ There exists an absolute constant $0<\theta'<1$ such
that if $nF_\la \leq \theta' C' \ln n $ and $(1-\theta') \ln n \geq \frac{3}{7}$ then,
$$\P(N_\la-nF_\la \geq (1-\theta') C' \ln n) \leq F_\la n^{-\gamma}.$$
\end{Lemma}
\begin{proof}
One takes $\theta'\in [0,1]$ 
such that $$\frac{(1-\theta')^2}{(2\theta'+1)}\geq \frac{4}{7}.$$
We use the Bernstein inequality that yields
$$\P(N_\la-nF_\la \geq (1-\theta')C' \ln n)
\leq\exp\left(-\frac{((1-\theta')C'\ln n)^2}{2(nF_\la+(1-\theta')C'\ln n/3)}\right)\leq
n^{-\frac{3C'(1-\theta')^2}{2(2\theta'+1)}}.$$
If $nF_\la \geq n^{-\gamma-1}$, since  $\frac{3 C'(1-\theta')^2}{2(2\theta'+1)} \geq
2 \gamma+2,$ the result is true.
If  $nF_\la \leq n^{-\gamma-1}$, using properties of Binomial random variables (see page 482 of \citet{sw}), for $n\geq 2$,
\begin{eqnarray*}
\P(N_\la-nF_\la \geq (1-\theta')C' \ln n)\leq\P(N_\la >(1-\theta')C'
\ln n)&\leq &\P(N_\la \geq 2)\\
&\leq&\frac{(1-F_{jk})C_n^2F_{jk}^2(1-F_{jk})^{n-2}}{1-3^{-1}(n+1)F_{jk}}\\
&\leq&\frac{n^2F_{jk}^2}{2(1-2^{-1}nF_{jk})}\\
&\leq & (nF_\la)^2
\end{eqnarray*}
and the result is true.
\end{proof}
Now, observe that if $|\hb_\la|>
\eta_{\la,\gamma}$ then
$$N_\la\geq C' \ln n.$$ Indeed, $|\hb_\la|>
\eta_{\la,\gamma}$ implies
$$\frac{C'\ln n
}n \norm{\p_\la}_\infty\leq |\hb_\la| \leq \frac{\norm{\p_\la}_\infty
N_\la}{n}.$$
So, if $n$ satisfies $(1-\theta') \ln n \geq \frac{3}{7}$, we set $\theta=\theta' C'\ln(n)$ and
$\mu=n^{-\gamma}$. In this case, Assumption (A3) is fulfilled since if $nF_\la \leq \theta' C' \ln n $
$$\P(|\hb_\la-\be_\la|>\kappa\eta_{\la,\ga}, |\hb_\la|>\eta_{\la,\ga})\leq \P(N_\la-nF_\la
\geq (1-\theta') C' \ln n) \leq F_\la n^{-\gamma}.$$

\noindent
Finally, if $n$ satisfies $(1-\theta') \ln n \geq \frac{3}{7}$, we can apply Theorem \ref{inegmodelsel} and we have:
\begin{equation}\label{init}
\frac{1-\kappa^2}{1+\kappa^2}\E\norm{\tb-\be}^2_{\ell_2}\leq \inf_{m\subset
  \Ga_n}\left\{\frac{1+\kappa^2}{1-\kappa^2}\sum_{(j,k)\not \in
  m}\be_\la^2+\frac{1-\kappa^2}{\kappa^2}\sum_{(j,k)\in
  m}\E(\hb_\la-\be_\la)^2+\sum_{(j,k)\in m}\E(\eta_{\la,\gamma}^2)\right\}+LD\sum_{(j,k)\in\Ga_n}F_\la.
\end{equation}
In addition,  there exists a constant  $K_1$ depending on  $p$, $\gamma$, $\kappa$, $c$,
$c'$ and on $\psi$ such that
\begin{equation}\label{init2}
LD\sum_{(j,k)\in\Ga_n}F_\la\leq K_1(\log(n))^{c'+1}n^{c-\frac{\kappa^2\ga}{q(2-\kappa)}-1}.
\end{equation}
Since $\ga>c$, one takes $\kappa<1$ and $q>1$ such that $c<\frac{\kappa^2\ga}{q(2-\kappa)}$ and as required by Theorem \ref{inegoraclelavraie}, the last term satisfies
$$LD\sum_{(j,k)\in\Ga_n}F_\la\leq \frac{K_2}{n},$$
where $K_2$ is a constant. Now we can derive the oracle inequality. Before evaluating the first term of (\ref{init}),
let us state the following lemma.
\begin{Lemma}\label{ecra}
We set for any $(j,k) \in \La$
$$D_{jk}=\int \p^2_{jk}(x) f(x)dx,$$
$$S_\p=\max\{\sup_{x\in \Supp(\phi)}|\phi(x)|,\sup_{x\in
 \Supp(\psi)}|\psi(x)|\}$$ and $$I_\p=\min\{\inf_{x\in
 \Supp(\phi)}|\phi(x)|,\inf_{x\in
 \Supp(\psi)}|\psi(x)|\}.$$
Using Appendix A, we define $\Theta_\p=\frac{S_\p^2}{I_\p^2}.$
For all $(j,k) \in \La$, we have the following result.
\begin{itemize}
\item[-] If
$F_\la\leq \Theta_\p\frac{\ln(n)}{n},$
then
$\be_\la^2\leq\Theta_\p^2D_{jk}\frac{\ln(n)}{n}.$\\
\item[-] If $F_\la>\Theta_\p\frac{\ln(n)}{n},$
then                                $\norm{\p_\la}_{\infty}\frac{\ln(n)}{n}\leq
\sqrt{\frac{D_{jk}\ln(n)}{n}}.$
\end{itemize}
\end{Lemma}

\begin{proof}
We assume that $j\geq 0$ (arguments are similar for
$j=-1$).\\ If
$F_\la\leq\Theta_\p\frac{\ln(n)}{n}$, we have
$$
|\be_\la|\leq S_\p 2^{\frac{j}{2}}F_\la
\leq S_\p2^{\frac{j}{2}}\sqrt{F_\la}\sqrt{\Theta_\p}\sqrt{\frac{\ln(n)}{n}}
\leq S_\p I^{-1}_\p\sqrt{\Theta_\p}\sqrt{\frac{D_{jk}\ln(n)}{n}}
\leq \Theta_\p\sqrt{\frac{D_{jk}\ln(n)}{n}},
$$
since
$D_{jk}\geq I^2_\p 2^jF_\la.$
For the second point, observe that
$$
\sqrt{\frac{D_{jk}\ln(n)}{n}}\geq 2^{\frac{j}{2}}I_\p\sqrt{\Theta_\p} \frac{\ln(n)}{n}=2^{\frac{j}{2}}S_\p\frac{\ln(n)}{n}\geq\norm{\psi_\la}_{\infty}\frac{\ln(n)}{n}
.
$$
\end{proof}
Now, for any $\delta>0$,
$$
\E(\eta_{\la,\gamma}^2)\leq (1+\delta) \frac{2 \ga \ln n}{n} \E({\widetilde{\si}}^2_{jk})+ (1+\delta^{-1})\left(\frac{2\ga \ln n}{3n}\right)^2\norm{\p_\la}_\infty^2.
$$
Moreover, $$\frac{\E({\widetilde{\si}}^2_{jk})}{n}\leq (1+\delta) \frac{D_{jk}}{n} +(1+\delta^{-1}) 8
\gamma \ln n \frac{\norm{\p_\la}_\infty^2}{n^2}.$$
So,
\begin{equation}\label{majothresh}
\E(\eta_{\la,\gamma}^2)\leq (1+\delta)^2 2 \ga \ln n   \frac{D_{jk}}{n} + \Delta(\delta) \left(\frac{\ga \ln n}{n}\right)^2\norm{\p_\la}_\infty^2,
\end{equation}
with $\Delta(\delta)$ a constant depending only on $\delta$. Now, we apply  (\ref{init}) with $$m=\left\{(j,k)\in\Ga_n:\quad\be_\la^2>\Theta_\p^2\frac{D_{jk}}{n}\ln n\ \right\},$$
so using Lemma \ref{ecra}, we can claim that for any $(j,k) \in m$, $F_\la>\Theta_\p\frac{\ln(n)}{n}$. Finally, since $\Theta_\p\geq 1$,
\begin{eqnarray*}
\E\norm{\tb-\be}^2_{\ell_2}&\leq& K_3\left( \sum_{(j,k)\in\Ga_n}\be_\la^2\indic_{\{\be_\la^2\leq\Theta_\p^2\frac{D_{jk}}{n}\ln
n\}}+\sum_{(j,k)\notin\Gamma_n}\be_\la^2\right)\\
&&+K_3\sum_{(j,k)\in\Ga_n}\left[
\frac{\ln n}{n} D_{jk} +\left(\frac{\ln n
}{n}\right)^2\norm{\p_\la}_\infty^2\right]\indic_{\left\{\be_\la^2>\Theta_\p^2\frac{D_{jk}}{n}\ln
n, \
F_\la>\Theta_\p\frac{\ln(n)}{n}\right\}}+\frac{K_4}{n}\\
&\leq& K_3\left[\sum_{(j,k)\in\Ga_n}\left(\be_\la^2\indic_{\left\{\be_\la^2\leq
\Theta_\p^2 \ln n\frac{D_{jk}}{n}\right\} }+2\ln n
\frac{D_{jk}}{n}\indic_{\left\{\be_\la^2>\Theta_\p^2 \ln
n\frac{D_{jk}}{n}\right\}}\right)+\sum_{(j,k)\notin\Gamma_n}\be_\la^2\right]+\frac{K_4}{n}\\
&\leq&2K_3\left[\sum_{(j,k)\in\Ga_n}\min\left(\be_\la^2,\Theta_\p^2\ln
n\frac{D_{jk}}{n}\right)+\sum_{(j,k)\notin\Gamma_n}\be_\la^2\right]+\frac{K_4}{n},
\end{eqnarray*}
where the constant $K_3$ depends on $\gamma$ and $c$ and $K_4$ depends on $\gamma$, $c$, $c'$ and on $\psi$. Finally, since
$$D_{jk}=\si^2_{jk}+\beta^2_{jk},$$
\begin{eqnarray*}
\E\norm{\tb-\be}^2_{\ell_2}&\leq&2K_3\left[\sum_{(j,k)\in\Ga_n}\min\left(\be_\la^2+\frac{\Theta_\p^2\ln
n}{n}\be_\la^2,\Theta_\p^2\ln
n\frac{\si^2_{jk}}{n}+\frac{\Theta_\p^2\ln
n}{n}\be_\la^2\right)+\sum_{(j,k)\notin\Gamma_n}\be_\la^2\right]+\frac{K_4}{n}\\
&\leq&2K_3\left[\sum_{(j,k)\in\Ga_n}\min\left(\be_\la^2,\Theta_\p^2\ln
n\frac{\si^2_{jk}}{n}\right)+\sum_{(j,k)\in\Ga_n}\frac{\Theta_\p^2\ln
n}{n}\be_\la^2+\sum_{(j,k)\notin\Gamma_n}\be_\la^2\right]+\frac{K_4}{n}\\
&\leq&2K_3\Theta_\p^2\left[\sum_{(j,k)\in\Ga_n}\min\left(\be_\la^2,\ln
n\frac{\si^2_{jk}}{n}\right)+\sum_{(j,k)\notin\Gamma_n}\be_\la^2\right]+2K_3\Theta_\p^2\norm{\be}_{\ell_2}\frac{\ln
n}{n}+\frac{K_4}{n}.
\end{eqnarray*}
Theorem \ref{inegoraclelavraie} is proved by using properties of the biorthogonal wavelet basis.

\subsection{Proof of Theorem \ref{lower}}
The first part is a direct application of Theorem \ref{inegoraclelavraie}. Now let us turn to the second part.
We recall that we consider $f=\indic_{[0,1]}$, the Haar basis and  for $j\geq 0$ and $k\in\Z$, we have:
$${\widetilde{\si}}^2_{jk}=\widehat{\si}^2_{jk}+2\norm{\p_{jk}}_\infty\sqrt{2\gamma\widehat{\si}^2_{jk}\frac{\ln n}{n}}+8\gamma \norm{\p_{jk}}_\infty^2 \frac{\ln n}{n}.$$
So, for any $0<\e<\frac{1-\gamma}{2}<\frac{1}{2}$,
$$
{\widetilde{\si}}^2_{jk}\leq(1+\e)\widehat{\si}^2_{jk}+2\gamma \norm{\p_{jk}}_\infty^2 \frac{\ln n}{n}\left(\e^{-1}+4\right)
.$$
Now,
\begin{eqnarray*}
\eta_{\la,\ga}&=&\sqrt{2\gamma {\widetilde{\si}}^2_{jk}\frac{\ln n}{n}} +\frac{2\norm{\p_{jk}}_\infty\gamma \ln n}{3n}\\
&\leq&\sqrt{2\gamma \frac{\ln n}{n}\left((1+\e)\widehat{\si}^2_{jk}+2\gamma \norm{\p_{jk}}_\infty^2 \frac{\ln n}{n}\left(\e^{-1}+4\right)\right)}+\frac{2\norm{\p_{jk}}_\infty\gamma \ln n}{3n}\\
&\leq&\sqrt{2\gamma (1+\e)\widehat{\si}^2_{jk} \frac{\ln n}{n}}+\frac{2\norm{\p_{jk}}_\infty\gamma \ln n}{n}\left(\frac{1}{3}+\sqrt{4+\e^{-1}}\right).
\end{eqnarray*}
Furthermore, using (\ref{sn})
$$ \widehat{\si}^2_{jk}=s_n-\frac{2}{n(n-1)}u_n,$$
and
$$\eta_{\la,\ga}\leq\sqrt{2\gamma (1+\e)\frac{\ln n}{n}s_n}+\sqrt{2\gamma (1+\e)\frac{\ln n}{n}\times\frac{2}{n(n-1)}|u_n|}+\frac{2\norm{\p_{jk}}_\infty\gamma \ln n}{n}\left(\frac{1}{3}+\sqrt{4+\e^{-1}}\right).$$
Using (\ref{concustats}), with probability larger than $1-6n^{-2}$,
$$|u_n|\leq U(2\ln n),$$
and, since $f=\indic_{[0,1]}$, we have $\si^2_{jk}\leq 1$ and
\begin{eqnarray*}
\frac{2}{n(n-1)}U(2\ln n) 
&\leq& C_1\frac{\log n}{n}+C_2\norm{\p_{jk}}_\infty^2\left(\frac{\log n}{n}\right)^{\frac{3}{2}},
\end{eqnarray*}
where 
$C_1$ and $C_2$ are universal constants. Finally, with probability larger than $1-6n^{-2}$,
$$\sqrt{2\gamma (1+\e)\frac{\ln n}{n}\times\frac{2}{n(n-1)}|u_n|}\leq\sqrt{2\gamma (1+\e)C_1}\frac{\ln n}{n}+\sqrt{2\gamma (1+\e)C_2}\norm{\p_{jk}}_\infty\left(\frac{\ln n}{n}\right)^{\frac{5}{4}}.$$
So, since $\gamma<1$, there exists $w(\e)$, only depending on $\e$ such that with probability larger than $1-6n^{-2}$,
\begin{eqnarray*}
\eta_{jk,\ga}&\leq&\sqrt{2\gamma (1+\e)\frac{\ln n}{n}s_n}+w(\e)\norm{\psi_{jk}}_\infty \frac{\ln n}{n}.
\end{eqnarray*}
Since $\norm{\psi_{jk}}_\infty=2^{j/2}$, we set
$$\widetilde{\eta_{jk,\ga}}=\sqrt{2\ga(1+\e) s_{n}\frac{\log n}{n}}+w(\e)\frac{2^{\frac{j}{2}}\log n}{n}$$
and $\eta_{jk,\ga}\leq \widetilde{\eta_{jk,\ga}}$ with probability larger than $1-6n^{-2}$. Then, since $f=\indic_{[0,1]}$, $\beta_{jk}=0$ for $j\geq 0$ and
\begin{eqnarray*}
s_{n}&=&\frac{1}{n}\sum_{i=1}^n\left(\psi_{jk}(X_i)-\beta_{jk}\right)^2\\
&=&\frac{2^j}{n}\sum_{i=1}^n\left(\indic_{X_i\in
[k2^{-j},(k+0.5)2^{-j}[}-\indic_{X_i\in
[(k+0.5)2^{-j},(k+1)2^{-j}[}\right)^2\\
&=&\frac{2^j}{n}\left(N^+_{jk}+N^-_{jk}\right),
\end{eqnarray*}
with $$N^+_{jk}=\sum_{i=1}^n\indic_{X_i\in
[k2^{-j},(k+0.5)2^{-j}[},\quad N^-_{jk}=\sum_{i=1}^n\indic_{X_i\in
[(k+0.5)2^{-j},(k+1)2^{-j}[}.$$
We consider $j$ such that
$$\frac{n}{(\log n)^\al}\leq 2^j<\frac{2n}{(\log n)^\al},\quad \al>1.$$
In particular, we have $$\frac{(\log n)^\al}{2}< n2^{-j}\leq (\log n)^\al.$$
Now,
$$\hb_\jk=\frac{1}{n}\sum_{i=1}^n\psi_{jk}(X_i)= \frac{2^{\frac{j}{2}}}{n}(N^+_{jk}-N^-_{jk}).$$
Hence,
\begin{eqnarray*}
\E(\norm{\tilde{f}_{n,\ga}-f}^2_2)&\geq        &        \sum_{k=0}^{2^j-1}
 \E\left(\hb_{jk}^2\indic_{|\hb_{jk}|\geq\eta_{jk,\ga}}\right)\\
&\geq        &        \sum_{k=0}^{2^j-1}
\E\left(\hb_{jk}^2\indic_{|\hb_{jk}|\geq\widetilde{\eta_{jk,\ga}}}\indic_{|u_n|\leq U(2\ln n)}\right)\\
&\geq & \sum_{k=0}^{2^j-1}
\frac{2^j}{n^2}\E\left((N^+_{jk}-N^-_{jk})^2 \indic_{|\hb_{jk}|\geq
\sqrt{2\ga (1+\e)s_{n}\frac{\log n}{n}}+w(\e)\frac{2^{j/2}\log n}{n}}\indic_{|u_n|\leq U(2\ln n)}\right).\\
&\geq & \sum_{k=0}^{2^j-1}
\frac{2^j}{n^2}\E\left((N^+_{jk}-N^-_{jk})^2
\indic_{\frac{2^{\frac{j}{2}}}{n}|N^+_{jk}-N^-_{jk}|\geq               \sqrt{2\ga(1+\e)
\frac{2^j}{n}\left(N^+_{jk}+N^-_{jk}\right)\frac{\log
n}{n}}+w(\e)\frac{2^{j/2}\log n}{n}}\indic_{|u_n|\leq U(2\ln n)}\right)\\
&\geq & \sum_{k=0}^{2^j-1}
\frac{2^j}{n^2}\E\left((N^+_{jk}-N^-_{jk})^2
\indic_{|N^+_{jk}-N^-_{jk}|\geq               \sqrt{2\ga(1+\e)
\left(N^+_{jk}+N^-_{jk}\right)\log
n}+w(\e)\log n}\indic_{|u_n|\leq U(2\ln n)}\right)\\
&\geq &
\frac{2^{2j}}{n^2}\E\left((N^+_{j1}-N^-_{j1})^2
\indic_{|N^+_{j1}-N^-_{j1}|\geq               \sqrt{2\ga(1+\e)
\left(N^+_{j1}+N^-_{j1}\right)\log
n}+w(\e)\log n}\indic_{|u_n|\leq U(2\ln n)}\right).
\end{eqnarray*}
Now, we consider a bounded sequence $(w_n)_n$ such that for any $n$, $w_n\geq w(\e)$ and such that $\frac{\sqrt{v_{nj}}}{2}$ is an integer with
$$v_{nj}=\left(\sqrt{4\ga(1+\e)\tilde\mu_{nj}\log(n)  }+w_n\log(n)\right)^2$$
and $\tilde\mu_{nj}$ is the largest integer smaller or equal to $n2^{-j-1}$. We have
$$v_{nj}\sim 4\ga(1+\e)\tilde\mu_{nj}\log n$$
and
$$\frac{(\log n)^\al}{4}-1<  n2^{-j-1}-1 <\tilde\mu_{nj}\leq n2^{-j-1}\leq \frac{(\log n)^\al}{2}.$$
So, if
$$N^+_{j1}=\tilde\mu_{nj}+\frac{1}{2}\sqrt{v_{nj}},\quad
N^-_{j1}=\tilde\mu_{nj}-\frac{1}{2}\sqrt{v_{nj}},$$
then
$$N^+_{j1}+N^-_{j1}=2\tilde\mu_{nj},\quad N^+_{j1}-N^-_{j1}=\sqrt{v_{nj}}=\sqrt{2\ga(1+\e)
\left(N^+_{j1}+N^-_{j1}\right)\log
n}+w_n\log n.$$
Finally,
\begin{eqnarray*}
\E(\norm{\tilde{f}_{n,\ga}-f}^2_2)&\geq &\frac{2^{2j}}{n^2}v_{nj}\P\left(
N^+_{j1}=\tilde\mu_{nj}+\frac{1}{2}\sqrt{v_{nj}},\quad
N^-_{j1}=\tilde\mu_{nj}-\frac{1}{2}\sqrt{v_{nj}},\quad |u_n|\leq U(2\ln n)\right)\\
&\geq&v_{nj}(\log
n)^{-2\al}\\
&&\hspace{0.3cm}\times\left[\P\left(
N^+_{j1}=\tilde\mu_{nj}+\frac{1}{2}\sqrt{v_{nj}},\quad
N^-_{j1}=\tilde\mu_{nj}-\frac{1}{2}\sqrt{v_{nj}}\right)-\P\left(|u_n|>U(2\ln n)\right)\right]\\
&\geq&v_{nj}(\log
n)^{-2\al}\left[\frac{n!}{l_{nj}!m_{nj}!(n-l_{nj}-m_{nj})!}p_j^{l_{nj}+m_{nj}}(1-2p_j)^{n-(l_{nj}+m_{nj})}-\frac{6}{n^2}\right],
\end{eqnarray*}
with
$$l_{nj}=\tilde\mu_{nj}+\frac{1}{2}\sqrt{v_{nj}},\quad m_{nj}=\tilde\mu_{nj}-\frac{1}{2}\sqrt{v_{nj}},$$
and
$$p_j=\int\indic_{
[k2^{-j},(k+0.5)2^{-j}[}(x) f(x)dx=\int \indic_{
[(k+0.5)2^{-j},(k+1)2^{-j}[}(x) f(x)dx=2^{-j-1}.$$
So,
\begin{eqnarray*}
\E(\norm{\tilde{f}_{n,\ga}-f}^2_2)&\geq &v_{nj}(\log
n)^{-2\al}\times\left[\frac{n!}{l_{nj}!m_{nj}!(n-2\tilde\mu_{nj})!}p_j^{2\tilde\mu_{nj}}(1-2p_j)^{n-2\tilde\mu_{nj}}-\frac{6}{n^2}\right].
\end{eqnarray*}
Now, let us study each term:
\begin{eqnarray*}
p_j^{2\tilde\mu_{nj}}&=&\exp\left(2\tilde\mu_{nj}\log(p_j)\right)\\
&=&\exp\left(2\tilde\mu_{nj}\log(2^{-j-1})\right),
\end{eqnarray*}
\begin{eqnarray*}
(1-2p_j)^{n-2\tilde\mu_{nj}}&=&\exp\left((n-2\tilde\mu_{nj})\log(1-2p_j)\right)\\
&=&\exp\left(-(n-2\tilde\mu_{nj})\left(2^{-j}+O_n(2^{-2j})\right)\right)\\
&=&\exp\left(-n2^{-j}\right)(1+o_n(1)),
\end{eqnarray*}
\begin{eqnarray*}
n!&=&n^ne^{-n}\sqrt{2\pi n}\;(1+o_n(1)),
\end{eqnarray*}
\begin{eqnarray*}
(n-2\tilde\mu_{nj})^{n-2\tilde\mu_{nj}}&=&\exp\left(\left(n-2\tilde\mu_{nj}\right)\log\left(n-2\tilde\mu_{nj}\right)\right)\\
&=&\exp\left(\left(n-2\tilde\mu_{nj}\right)\left(\log n+\log\left(1-\frac{2\tilde\mu_{nj}}{n}\right)\right)\right)\\
&=&\exp\left(\left(n-2\tilde\mu_{nj}\right)\log n-\frac{2\tilde\mu_{nj}\left(n-2\tilde\mu_{nj}\right)}{n}\right)(1+o_n(1))\\
&=&\exp\left(n\log n-2\tilde\mu_{nj}-2\tilde\mu_{nj}\log n\right)(1+o_n(1)).
\end{eqnarray*}
Then,
\begin{eqnarray*}
\frac{n!}{(n-2\tilde\mu_{nj})!}p_j^{2\tilde\mu_{nj}}(1-2p_j)^{n-2\tilde\mu_{nj}}
&=&\frac{e^{n-2\tilde\mu_{nj}}}{e^n}\times
\frac{n^n}{(n-2\tilde\mu_{nj})^{n-2\tilde\mu_{nj}}}\times  p_j^{2\tilde\mu_{nj}}(1-2p_j)^{n-2\tilde\mu_{nj}}\times
(1+o_n(1))\\
&=&\exp\left(-2\tilde\mu_{nj}\right)\times
\frac{\exp\left(n\log n\right)}{(n-2\tilde\mu_{nj})^{n-2\tilde\mu_{nj}}}\times  p_j^{2\tilde\mu_{nj}}(1-2p_j)^{n-2\tilde\mu_{nj}}\times
(1+o_n(1))\\
&=&\exp\left(-2\tilde\mu_{nj}\right)\times\frac{\exp\left(n\log n+2\tilde\mu_{nj}\log(2^{-j-1})-n2^{-j}\right)}{\exp\left(n\log n-2\tilde\mu_{nj}-2\tilde\mu_{nj}\log n\right)}(1+o_n(1))\\
&=&\exp\left(2\tilde\mu_{nj}\log n+2\tilde\mu_{nj}\log(2^{-j-1})-n2^{-j}\right)(1+o_n(1)).
\end{eqnarray*}
It remains to evaluate $l_{nj}!\times m_{nj}!$
\begin{eqnarray*}
l_{nj}!\times      m_{nj}!&=&\left(\frac{l_{nj}}{e}\right)^{l_{nj}}\left(\frac{m_{nj}}{e}\right)^{m_{nj}}\sqrt{2\pi
l_{nj}}\sqrt{2\pi m_{nj}}(1+o_n(1))\\
&=&\exp\left(l_{nj}\log l_{nj}+m_{nj}\log m_{nj}-2\tilde\mu_{nj}\right)\times 2\pi\tilde\mu_{nj}(1+o_n(1)).
\end{eqnarray*}
If we set
$$x_{nj}=\frac{\sqrt{v_{nj}}}{2\tilde\mu_{nj}}=o_n(1),$$
then
$$l_{nj}=\tilde\mu_{nj}+\frac{\sqrt{v_{nj}}}{2}=\tilde\mu_{nj}(1+x_{nj}),$$
$$m_{nj}=\tilde\mu_{nj}-\frac{\sqrt{v_{nj}}}{2}=\tilde\mu_{nj}(1-x_{nj}),$$
and using that
\begin{eqnarray*}
(1+x_{nj})\log(1+x_{nj})&=&(1+x_{nj})\left(x_{nj}-\frac{x_{nj}^2}{2}+\frac{x_{nj}^3}{3}+O(x_{nj}^4)\right)\\
&=&x_{nj}-\frac{x_{nj}^2}{2}+\frac{x_{nj}^3}{3}+x_{nj}^2-\frac{x_{nj}^3}{2}+O(x_{nj}^4)\\
&=&x_{nj}+\frac{x_{nj}^2}{2}-\frac{x_{nj}^3}{6}+O(x_{nj}^4)
\end{eqnarray*}
\begin{eqnarray*}
l_{nj}\log l_{nj}&=&\tilde\mu_{nj}(1+x_{nj})\log\left(\tilde\mu_{nj}(1+x_{nj})\right)\\
&=&\tilde\mu_{nj}(1+x_{nj})\log(1+x_{nj})+\tilde\mu_{nj}(1+x_{nj})\log\left(\tilde\mu_{nj}\right)\\
&=&\tilde\mu_{nj}\left(x_{nj}+\frac{x_{nj}^2}{2}-\frac{x_{nj}^3}{6}+O(x_{nj}^4)\right)+\tilde\mu_{nj}(1+x_{nj})\log\left(\tilde\mu_{nj}\right).
\end{eqnarray*}
Similarly,
\begin{eqnarray*}
m_{nj}\log m_{nj}&=&\tilde\mu_{nj}\left(-x_{nj}+\frac{x_{nj}^2}{2}+\frac{x_{nj}^3}{6}+O(x_{nj}^4)\right)+\tilde\mu_{nj}(1-x_{nj})\log\left(\tilde\mu_{nj}\right).
\end{eqnarray*}
So,
\begin{eqnarray*}
l_{nj}\log l_{nj}+m_{nj}\log m_{nj}&=&\tilde\mu_{nj}\left(x_{nj}^2+O(x_{nj}^4)\right)+2\tilde\mu_{nj}\log\left(\tilde\mu_{nj}\right)\\
&\leq&\tilde\mu_{nj}x_{nj}^2+2\tilde\mu_{nj}\log(n2^{-j-1})+O(\tilde\mu_{nj}x_{nj}^4).
\end{eqnarray*}
Since
$$\tilde\mu_{nj}x_{nj}^2=\frac{v_{nj}}{4\tilde\mu_{nj}}\sim\ga(1+\e)\log n,
$$
for $n$ large enough,
$$
\tilde\mu_{nj}x_{nj}^2+O(\tilde\mu_{nj}x_{nj}^4)\leq(\ga+2\e)\log n
$$
and
$$l_{nj}\log l_{nj}+m_{nj}\log m_{nj}\leq (\ga+2\e)\log n+2\tilde\mu_{nj}\log(n2^{-j-1}).$$
Finally,
\begin{eqnarray*}
l_{nj}!\times      m_{nj}!
&=&\exp\left(l_{nj}\log l_{nj}+m_{nj}\log m_{nj}-2\tilde\mu_{nj}\right)\times 2\pi\tilde\mu_{nj}(1+o_n(1))\\
&\leq&\exp\left((\ga+2\e)\log n+2\tilde\mu_{nj}\log(n2^{-j-1})-2\tilde\mu_{nj}\right)\times 2\pi\tilde\mu_{nj}(1+o_n(1)).
\end{eqnarray*}
we derive that 
\begin{eqnarray*}
\E(\norm{\tilde{f}_{n,\ga}-f}^2_2)&\geq &
v_{nj}(\log
n)^{-2\al}\times\left[\frac{n!}{l_{nj}!m_{nj}!(n-2\tilde\mu_{nj})!}p_j^{2\tilde\mu_{nj}}(1-2p_j)^{n-2\tilde\mu_{nj}}-\frac{6}{n^2}\right]\\
&\geq&v_{nj}(\log
n)^{-2\al}\times\left[
\frac{\exp\left(2\tilde\mu_{nj}\log n+2\tilde\mu_{nj}\log(2^{-j-1})-n2^{-j}\right)}{\exp\left((\ga+2\e)\log n+2\tilde\mu_{nj}\log(n2^{-j-1})-2\tilde\mu_{nj}\right)\times 2\pi\tilde\mu_{nj}}-\frac{6}{n^2}\right](1+o_n(1))
\\
&\geq&
v_{nj}(\log
n)^{-2\al}\times\left[
\frac{\exp\left(-(\ga+2\e)\log n-2\right)}{2\pi\tilde\mu_{nj}}-\frac{6}{n^2}\right](1+o_n(1))\\
\end{eqnarray*}
So there exists $C_1$ and $C_2$ two positive constants such that, for $n$ large enough
$$\E(\norm{\tilde{f}_{n,\ga}-f}^2_2) \geq C_1 (\log n)^{1-\alpha}\left[C_2 \frac{n^{-(\ga+2\e)}}{(\log n)^\alpha}-\frac{6}{n^2}\right].$$
As $0<\ga+2\e<1$, there exists a positive constant $\delta<1$ such that
$$\E(\norm{\tilde{f}_{n,\ga}-f}^2_2) \geq \frac{1}{n^{\delta}}(1+o_n(1)).$$
This concludes the proof of Theorem \ref{lower}.

\vspace{2cm}

\noindent{\bf Ackowledgment}: The authors  acknowledge the support of the  French Agence Nationale
de la Recherche  (ANR), under grant ATLAS (JCJC06\_137446) ''From Applications
to Theory in Learning and Adaptive Statistics''. We also warmly thank Rebecca Willett for her very smart program, and both A. Antoniadis and L. Birg\'e for a wealth of advice and encouragement.

\bibliographystyle{agms}

\end{document}